%% file: hard-core-new.tex
\documentclass[11pt]{amsart}

\usepackage{amsmath, amsthm, amsopn, amssymb, enumerate, tikz}

\newtheorem{thm}{Theorem}[section]
\newtheorem{lemma}[thm]{Lemma}
\newtheorem{prop}[thm]{Proposition}

\newtheorem{cl}{Claim}

\newtheorem*{claim}{Claim}

\theoremstyle{definition}

\newtheorem*{remark}{Remark}

\newcommand{\Bin}{\mathrm{Bin}}

\newcommand{\Bvis}{\partial^{\mathrm{vis}}}
\newcommand{\Break}{\mathrm{Break}}
\renewcommand{\NG}{N_{\Gamma}}
\newcommand{\PG}{P_{\Gamma}}
\newcommand{\RG}{R_{\Gamma}}
\newcommand{\MCut}{\mathrm{MCut}}
\newcommand{\OMCut}{\mathrm{OMCut}}
\newcommand{\Shift}{\mathrm{Shift}}
\newcommand{\dist}{\mathrm{dist}}

\newcommand{\cA}{\mathcal{A}}
\newcommand{\cN}{\mathcal{N}}
\newcommand{\cP}{\mathcal{P}}
\newcommand{\E}{\mathcal{E}}
\newcommand{\Ex}{\mathbb{E}}
\newcommand{\Fo}{\mathcal{F}^{\mathrm{odd}}}

\newcommand{\mune}{\mu_n^{\mathrm{even}}}
\newcommand{\muno}{\mu_n^{\mathrm{odd}}}

\newcommand{\muo}{\mu^{\mathrm{odd}}}

\newcommand{\N}{\mathbb{N}}

\newcommand{\Ve}{V^{\mathrm{even}}}
\newcommand{\Vo}{V^{\mathrm{odd}}}
\newcommand{\Vva}{V^{\mathrm{vac}}}

\newcommand{\Z}{\mathbb{Z}}

\title{Odd cutsets and the hard-core model on $\Z^d$}
\date{\today}

\author{Ron Peled}
\thanks{Research of R.P. supported by a Marie Curie Reintegration Grant SPTRF from the Commission of the European Communities. Research of W.S.~supported by ERC Advanced Grant DMMCA}
\address{Ron Peled\hfill\break
Tel Aviv University\\
School of Mathematical Sciences\\
Tel Aviv, 69978, Israel.}
\email{peledron@post.tau.ac.il}
\urladdr{http://www.math.tau.ac.il/~peledron}

\author{Wojciech Samotij}
\address{Wojciech Samotij\hfill\break
Tel Aviv University\\
School of Mathematical Sciences\\
Tel Aviv, 69978, Israel;
and Trinity College\\
Cambridge CB2 1TQ, UK.}
\email{samotij@post.tau.ac.il}
\urladdr{http://www.math.tau.ac.il/~samotij}

\begin{document}

\maketitle

\begin{abstract}
  We consider the hard-core lattice gas model on $\Z^d$ and investigate its phase structure in high dimensions. We prove that when the intensity parameter exceeds $Cd^{-1/3}(\log d)^2$, the model exhibits multiple hard-core measures, thus improving the previous bound of $Cd^{-1/4}(\log d)^{3/4}$ given by Galvin and Kahn. At the heart of our approach lies the study of a certain class of edge cutsets in $\Z^d$, the so-called odd cutsets, that appear naturally as the boundary between different phases in the hard-core model. We provide a refined combinatorial analysis of the structure of these cutsets yielding a quantitative form of concentration for their possible shapes as the dimension $d$ tends to infinity. This analysis relies upon and improves previous results obtained by the
  first author.
\end{abstract}

\section{Introduction}

\label{sec:introduction}

The {\em hard-core model} (short for {\em hard-core lattice gas model}) was originally introduced in statistical mechanics as a simple mathematical model of a gas whose particles have non-negligible size and cannot overlap, see~\cite{Do, GiLeMa}. Later, the model was rediscovered in operations research in the context of certain communication networks, see~\cite{Ke85, Ke91}. The model has also attracted interest in ergodic theory where it is known by the name ``the golden mean subshift'' \cite{S90}, since its topological entropy on the one-dimensional lattice is the logarithm of the golden ratio $(1+\sqrt{5})/2$.

Let $G$ be a finite graph with vertex set $V$ and let $\lambda$ be a positive real. A {\em configuration} $\omega \in \{0, 1\}^V$ is called {\em feasible} if it is the characteristic function of some independent set in $G$, i.e., if for no adjacent pair $\{v_1, v_2\}$ of vertices of $G$, it satisfies $\omega(v_1) = \omega(v_2) = 1$. The {\em hard-core model on $G$ with activity $\lambda$} is the probability measure $\mu$ on $\{0, 1\}^V$ defined by
\[
\mu(\omega) =
\begin{cases}
  Z^{-1} \prod_{v \in V} \lambda^{\omega(v)} & \text{if $\omega$ is feasible}, \\
  0 & \text{otherwise},
\end{cases}
\]
where $Z$ is the appropriate normalizing constant (called the {\em
partition function}) which makes $\mu$ a probability measure, i.e.,
$Z = \sum_\omega \prod_v \lambda^{\omega(v)}$, where $\omega$ ranges
over the set of all feasible configurations. In other words,
$\mu(\omega)$ is proportional to $\lambda^{|\{v \colon \omega(v) =
1\}|}$ if $\omega$ is a feasible configuration and $\mu(\omega) = 0$
otherwise.

If $G$ is an infinite graph, then we call a probability measure on $\{0,1\}^V$ a {\em hard-core measure} if for each finite $W \subseteq V$, every $\omega_1 \in \{0,1\}^W$, and $\mu$-a.e. $\omega_2 \in \{0,1\}^{V \setminus W}$, we have
\[
\mu\Bigl(\omega|_W = \omega_1 \;\Bigm\vert\; \omega|_{V \setminus W} = \omega_2 \Bigr) =
\begin{cases}
  Z_W^{-1} \prod_{v \in W} \lambda^{\omega(v)} & \text{if $\omega_1 \sqcup \omega_2$ is feasible}, \\
  0 & \text{otherwise},
\end{cases}
\]
where $\omega_1 \sqcup \omega_2$ is the configuration on $V$ that agrees with $\omega_1$ on $W$ and with $\omega_2$ on $V \setminus W$. A standard compactness argument shows that for any infinite but locally finite graph $G$ and any $\lambda$, there exists at least one hard-core measure on $G$ with activity $\lambda$. One of the main questions in the study of the hard-core model is to decide, for given $G$ and $\lambda$, whether there is a unique or multiple hard-core measures on $G$ with activity $\lambda$. An important contribution to this problem was made by van den Berg and Steif~\cite{vdBS94}, continuing previous work of van den Berg \cite{vdB93}, who showed that for any connected infinite graph\footnote{More precisely, a countably infinite, locally finite, connected graph.} $G$, there is a unique hard-core measure with activity $\lambda$ whenever $\lambda<\frac{p_c(G)}{1-p_c(G)}$. Here, $p_c(G)$ stands for the site percolation threshold for the graph $G$.

The most-studied case of a hard-core model is that when $G$ is the
nearest-neighbor graph of the integer lattice $\Z^d$, i.e., the
graph on the vertex set $\Z^d$ in which two vertices are adjacent if
and only if their $L^1$-distance is equal to $1$. See
Figure~\ref{fig:simulation} below for a simulation of the model in
two dimensions. The above-mentioned result of van den Berg and
Steif, combined with the simple lower bound $p_c(\Z^d)\ge
\frac{1}{2d-1}$, proves that if $\lambda<\frac{1}{2d-2}$ there is a
unique hard-core measure with activity $\lambda$ on $\Z^d$. The
seminal result of Dobrushin~\cite{Do} says that when $d \geq 2$ and
$\lambda$ is sufficiently large (depending on $d$), then $\Z^d$
admits multiple hard-core measures with activity $\lambda$;
Dobrushin's result was later rediscovered by Louth~\cite{Lo}. The
lower bound on $\lambda$ proved by Dobrushin is rather weak (see the
discussion in~\cite{GaKa}) as it grows with $d$, quite in contrast
with the popular belief that for a given $\lambda$, the existence of
multiple hard-core measures in dimension $d$ should imply the
existence of multiple hard-core measures in all higher dimensions.
Almost 40 years had passed since Dobrushin published his result
before Galvin and Kahn~\cite{GaKa} proved that the threshold
activity that implies the existence of multiple hard-core measures
on $\Z^d$ tends to $0$ as $d \to \infty$.

\begin{thm}[\cite{GaKa}]
  \label{thm:GaKa}
  There exist constants $d_0$ and $C$ such that if $d \geq d_0$ and $\lambda \geq Cd^{-1/4}(\log d)^{3/4}$, then the graph $\Z^d$ has multiple hard-core measures with activity $\lambda$.
\end{thm}

The bound in Theorem~\ref{thm:GaKa} is undoubtedly not best possible. It is commonly believed that multiple hard-core measures should appear when the activity $\lambda$ is either at $\Theta(1/d)$ or at $\Theta(\log d/d)$. A natural question in this line of research is that of the existence of a {\em critical activity} $\lambda_c(d)$ such that there are multiple hard-core measures as soon as $\lambda > \lambda_c(d)$ and a unique measure if $\lambda < \lambda_c(d)$. Even though it is believed that in the case of $\Z^d$, such critical activity exists, so far it has not been proved or disproved. Interestingly, it was shown in~\cite{BrHaWi} that there are graphs for which there is no such critical activity. Following~\cite{GaKa}, even if the hard-core model on $\Z^d$ does not have a critical activity, one can still define a similar quantity $\lambda(d)$ to be the supremum of those activities $\lambda$ for which there is a unique hard-core measure. With this notation at hand, one can rephrase Theorem~\ref{thm:GaKa} as $\lambda(d) = O(d^{-1/4}(\log d)^{3/4})$ as $d \to \infty$.

The aim of this paper is twofold. First, using the machinery of
\emph{odd cutsets} developed by the first author in~\cite{Pe}, we
give what we think is a shorter and more transparent proof of (a
slightly weaker version of) the result of Galvin and Kahn.

\begin{thm}
  \label{thm:GaKa-weak}
  There are constants $d_0$ and $C$ such that if $d \geq d_0$, then
  \[
  \lambda(d) \leq Cd^{-1/4}(\log d)^2.
  \]
\end{thm}

Second, we prove a refined version of one of the main results
of~\cite{Pe}, the so-called {\em interior approximation} theorem for
odd cutsets in $\Z^d$ and use it to improve the bound in
Theorem~\ref{thm:GaKa}.

\begin{thm}
  \label{thm:main}
  There are constants $d_0$ and $C$ such that if $d \geq d_0$, then
  \[
  \lambda(d) \leq Cd^{-1/3}(\log d)^2.
  \]
\end{thm}

We would like to point out that, mimicking the general strategy in~\cite{Pe}, we try to separate the probabilistic and geometric parts of the argument as much as possible. First of all, we hope that this makes our argument better structured and more easily comprehensible. Secondly, we aim to stress our belief that a further refinement of the geometric part will improve the bound on $\lambda(d)$. Last but not least, this separation allows us to reuse most of the proof of Theorem~\ref{thm:GaKa-weak} in the proof of Theorem~\ref{thm:main}.

\section{Preliminaries}

\label{sec:preliminaries}

\subsection{Notation}

\noindent
{\bf General.}
For a set $X$, we will denote the power set (the set of all subsets) of $X$ by $\cP(X)$ and for a positive integer $k$, we will abbreviate $\{1, \ldots, k\}$ by $[k]$. We always write $\log$ for the natural logarithm.

\smallskip
\noindent
{\bf Graph theory.}
Given a graph $H$, we will denote its vertex and edge sets by $V(H)$ and $E(H)$, respectively. For the sake of clarity, for two vertices $v, w \in V(H)$, we will sometimes write $vw$ for the unordered pair $\{v, w\}$. If $W \subseteq V(H)$, then we will write $H[W]$ for the graph induced by $H$ on $W$ and $H \setminus W$ for $H[V(H) \setminus W]$. If $F \subseteq E(H)$, then we will write $H \setminus F$ for the graph obtained from $H$ by removing from it all edges in $F$. For every $v \in V(H)$, the {\em neighborhood} of $v$, denoted $N(v)$, is the set of all vertices that are adjacent to $v$; for a set $S \subseteq V(H)$, we let $N_H(S) = \bigcup_{v \in S} N(v)$. The distance between two vertices $v, w \in V(H)$, denoted $\dist_H(v,w)$, is the length of the shortest path connecting $v$ and $w$ in $H$ or $\infty$ if there is no such path. A set $W \subseteq V(H)$ is {\em connected} if each pair of vertices in $W$ is connected by a path, i.e., if $\dist_H(v,w) < \infty$ for every $v, w \in W$. Sometimes we will drop the subscript $H$ to keep things less cluttered, provided that the graph $H$ is clear from the context. The {\em connected component} of a vertex $v$ (a connected set $W$) is the largest connected subset of $V(H)$ that contains the vertex $v$ (the set $W$). For a positive integer $k$, the {\em $k$th power of $H$}, denoted $H^k$, is the graph on the vertex set $V(H)$ whose two vertices $v$ and $w$ are adjacent if and only if $\dist_H(v, w) \leq k$. Finally, we will write $\Delta(H)$ for the maximum degree of $H$.

\smallskip
\noindent
{\bf The graph $\Z^d$.}
Since this graph is clearly connected and bipartite, it is natural to refer to its two color classes as the even and the odd vertices, denoted $\Ve$ and $\Vo$ respectively. The set $\Ve$ ($\Vo$) consists of those lattice points whose coordinates sum up to an even (odd) number. Recall that two vertices $v, w \in \Z^d$ are adjacent if and only if $w = v + f$ for some $f \in \Z^d$ whose $L^1$-norm is equal to $1$. It is therefore convenient to denote all such vectors $f$, i.e., all $\{-1,0,1\}$-vectors with exactly one non-zero coordinate, by $f_1, \ldots, f_{2d}$. For example, observe that $N(v) = \{v + f_j \colon j \in [2d]\}$ for every $v \in \Z^d$. Finally, when a feasible configuration $\omega \in \{0,1\}^{\Z^d}$ is clear from the context, $\Vva$ will denote the set of all {\em vacant} vertices, i.e., $\Vva = \{v \in \Z^d \colon \omega(v) = 0\}$. If a vertex is not vacant, then we say that it is {\em occupied}.

\subsection{Finitized version of the problem}

As observed by Galvin and Kahn~\cite{GaKa}, the problem of showing
the existence of multiple hard-core measures can be finitized as
follows. Let $G_n$ be the subgraph of $\Z^d$ induced on the set of
lattice points in $[-n,n]^d$ and define the boundary $B_n$ of this
subgraph by $B_n = [-n,n]^d \setminus [-(n-1),n-1]^d$. Let $\mu_n$
be the (unique) hard-core measure on $G_n$ with activity $\lambda$,
let $\mune$ be $\mu_n$ conditioned on the event $\omega(v) = 1$ for
all $v \in B_n \cap \Ve$, and define $\muno$ accordingly. As shown
in~\cite{GaKa}, there are positive constants $c$ and $C$ such that
the following holds. If for every fixed $x \in \Z^d$ and
sufficiently large $n$,
\begin{align*}
  \mune( \omega(x) = 1) & < c/d \quad \text{if $x \in \Vo$} \quad \text{and} \\
  \muno( \omega(x) = 1) & < c/d \quad \text{if $x \in \Ve$},
\end{align*}
then there are multiple hard-core measures on $\Z^d$ with activity $\lambda$, provided that $\lambda > C/d$. In view of the above, Theorem~\ref{thm:main} will be a direct consequence of the next theorem, whose proof is the main part of this paper.

\begin{thm}
  \label{thm:finitized}
  There exist constants $d_0$ and $C$ such that if $d \geq d_0$, $\lambda \geq Cd^{-1/3}(\log d)^2$, $n \geq 2$, and $x$ is an arbitrary even vertex of $G_n$, then
  \[
  \muno( \omega(x) = 1 ) < c/d.
  \]
  Moreover, the same result holds when the roles of even and odd vertices are reversed.
\end{thm}

\begin{remark}
  We will first prove Theorem~\ref{thm:finitized} under the stronger assumption that $\lambda \geq Cd^{-1/4}(\log d)^2$, which implies Theorem~\ref{thm:GaKa-weak}. Later, we will refine our approach and prove it for the stated regime.
\end{remark}

\subsection{Tools}

In this section, we collect two auxiliary lemmas that will be used in the proof of our main result. Both are fairly standard and can be surely found in the literature. The proofs are given only for the sake of completeness.

\begin{lemma}
  \label{lemma:connected-sets}
  Let $M$ be an integer, let $H$ be an arbitrary graph, and let $v \in V(H)$. The number of connected sets $E \subseteq V(H)$ such that $v \in E$ and $|E| = M$ does not exceed $(\Delta(H))^{2M-2}$.
\end{lemma}
\begin{proof}
  For every such $E$, we fix an arbitrary spanning tree $T_E$ of $E$. Starting from $v$, we perform a depth-first search on $T_E$, starting and ending at $v$ and passing through every edge exactly twice. Since every spanning tree of $E$ has exactly $M-1$ edges, the number of possibilities for such a walk (and hence for $E$) is not larger than the number of walks of length $2M-2$ in $H$ that start at $v$.
\end{proof}

Our second lemma formalizes the following intuition. If an event $A$ in some discrete probability space $(X, \mu)$ admits an expanding transformation $T \colon A \to \cP(X)$, i.e., a map for which $\mu(T(a))$ is much larger than $\mu(a)$ for every $a \in A$, then $\mu(A)$ is small, provided that no $x \in X$ appears in the image of $T$ too many times.

\begin{lemma}
  \label{lemma:double-counting}
  Let $X$ be a finite set, let $A \subseteq X$, and let $\mu$ be a probability measure on $X$. Suppose that there are positive numbers $p$ and $q$ and a mapping $T \colon A \to \cP(X)$ such that for each $a \in A$ and each $x \in X$,
  \[
  \mu(T(a)) \geq q \cdot \mu(a) \quad \text{and} \quad |\{b \in A \colon x \in T(b)\}| \leq p.
  \]
  Then $\mu(A) \leq p/q$.
\end{lemma}
\begin{proof}
  Our assumptions easily imply that
  \begin{align*}
    \mu(A) & = \sum_{a \in A}\mu(a) \leq \frac{1}{q} \cdot \sum_{a \in A} \mu(T(a)) = \frac{1}{q} \cdot \sum_{a \in A} \sum_{x \in T(a)} \mu(x) \\
    & = \frac{1}{q} \cdot \sum_{x \in X} |\{a \in A \colon x \in T(a)\}| \cdot \mu(x) \leq \frac{p}{q} \cdot \sum_{x \in X} \mu(x) = \frac{p}{q}.
    \qedhere
  \end{align*}
\end{proof}

\section{Outline of the argument}

\label{sec:outline}

Fix integers $n$ and $d$ and recall that $G_n$ is the subgraph of
$\Z^d$ induced on the set of lattice points in $[-n,n]^d$. We start
by observing that if the odd boundary vertices of $G_n$ are
occupied, and $x$ is an even vertex which is occupied, there must be
an "outermost" surface in $G_n$ in which the pattern "flips" from
odd/occupied to even/occupied (see Figure~\ref{fig:simulation}).
Formally, this surface is defined in
Section~\ref{sec:locating_cutset} as a set of edges $\Gamma$ having
the following properties:
\begin{enumerate}[(i)]
\item
  $\Gamma$ forms a minimal cutset separating $B_n$ from $x$. In other words, $\Gamma$ partitions $G_n$ into two connected components, the component of the boundary, denoted by $A_0(\Gamma)$, and the component of $x$, denoted by $A_1(\Gamma)$.
\item
  The vertices on the outer boundary of $\Gamma$ are even and vacant and the vertices on the inner boundary of $\Gamma$ are odd and vacant.
\end{enumerate}
Edge cutsets with the above properties play a prominent role in our analysis and we term them $\OMCut$ (for odd minimal edge cutsets).

In what follows, let $x$ be a fixed even vertex of $G_n$ and let $\Omega$ be our ``bad'' event, i.e., the set of all feasible configurations with all odd boundary vertices (i.e., the vertices in the set $\Vo \cap B_n$) and $x$ occupied. For a given configuration $\omega\in\Omega$ we term the above $\Gamma$ as $\Break(\omega)$. A simple consequence of the above properties and the fact that $x$ is even is that $|\Break(\omega)|\ge 2d(2d-1)$ (Proposition~\ref{prop:Gamma-size}).

\begin{figure}[h]
  \centering
  \begin{tabular}{ccc}
    \input{fig-simulation-41x41-1.0} & \hspace{0.2in} & \input{fig-simulation-41x41-5.0}
  \end{tabular}
  \caption{Typical configurations with odd boundary conditions for $n=20$, $d=2$, and activities $\lambda = 1$ and $\lambda = 5$, respectively. Odd and even occupied nodes are represented by black and gray circles, respectively. Even occupied nodes are surrounded by their corresponding $\Break$s. Simulated using coupling from the past~\cite{HN98}.}
  \label{fig:simulation}
\end{figure}
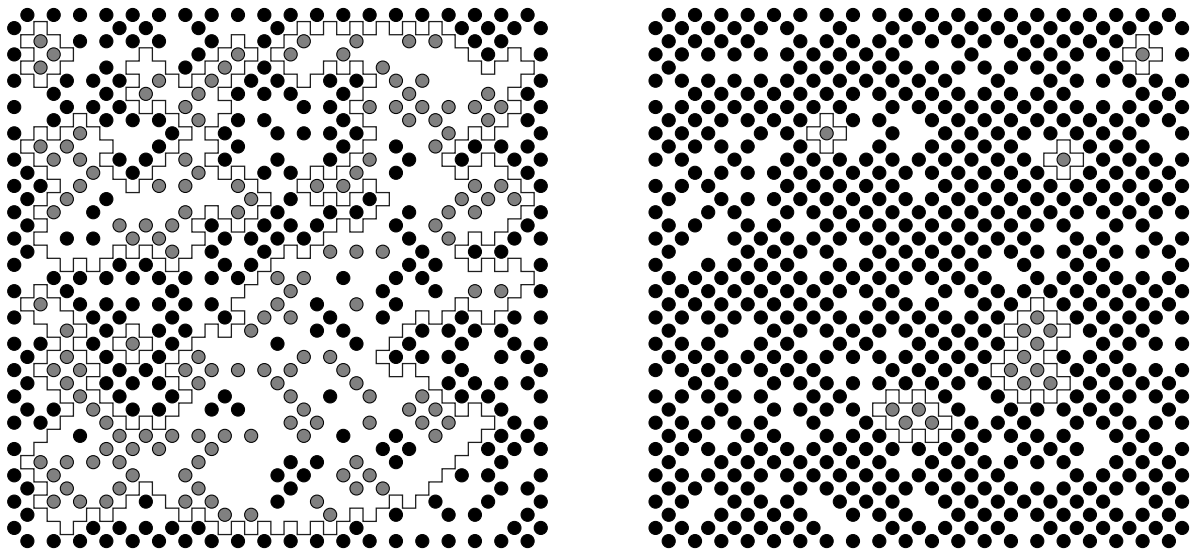

Our next observation is that applying the following "shift
transformation" to $\omega\in\Omega$ yields a feasible configuration
(see Figure~\ref{fig:shift}):
\begin{equation*}
  \Shift_j(\omega)(v) =
  \begin{cases}
    \omega(v + f_j) & \text{if $v \in A_1(\Break(\omega))$}, \\
    \omega(v) & \text{otherwise}.
  \end{cases}
\end{equation*}
A similar transformation was used by Galvin and Kahn~\cite{GaKa}. The key properties of this transformation are:
\begin{enumerate}[(i)]
\item
  $\Shift_j$ is measure preserving, i.e., $\muo(\Shift_j(\omega)) = \muo(\omega)$.
\item
  Denoting $\omega'=\Shift_j(\omega)$, every vertex in the set $\{v\in A_1(\Break(\Gamma)) \colon (v,v+f_j)\in\Break(\omega)\}$ is surrounded by vacant vertices in $\omega'$. Thus an arbitrary modification of $\omega'$ on these vertices yields a feasible configuration.
\end{enumerate}
These two properties and an application of Lemma~\ref{lemma:double-counting} imply that for every fixed $\Gamma\in\OMCut$,
\begin{equation}
  \label{ineq:Break_is_Gamma}
  \muo(\Break(\omega)=\Gamma) \leq (1+\lambda)^{-|\Gamma|/2d}.
\end{equation}

\begin{figure}[h]
  \centering
  \begin{tabular}{ccc}
    \input{fig-pre-shift} & \hspace{0.5in} & \input{fig-post-shift}
  \end{tabular}
  \caption{The shift transformation. Odd and even occupied nodes are represented by black and gray circles, respectively. The transformation $T_1$ shifts the occupied nodes inside the cutset $\Break(\omega)$ to the left.
  In the new configuration all vertices having an edge of $\Break(\omega)$ on their right (represented by white circles) necessarily have no occupied neighbors and can thus be set occupied or vacant arbitrarily.}
  \label{fig:shift}
\end{figure}
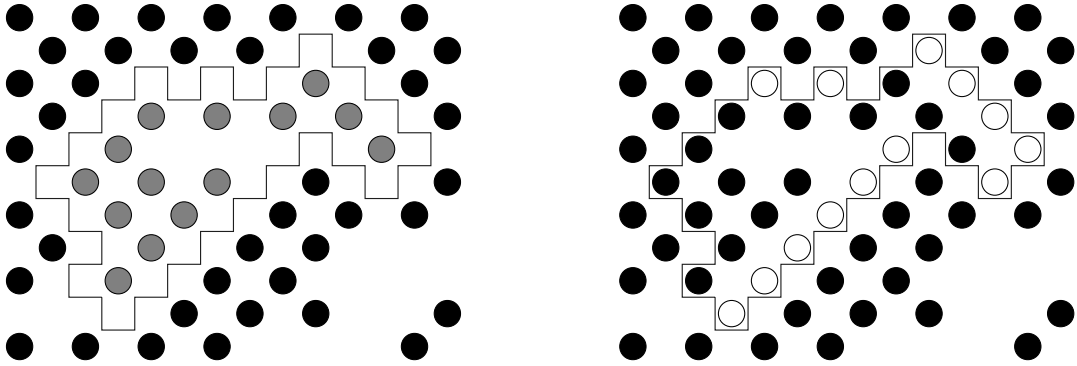

At this point one is tempted to conclude the proof of
Theorem~\ref{thm:finitized} by the union bound over all possible
$\Gamma\in\OMCut$. This approach unfortunately fails since the
number of $\Gamma\in\OMCut$ with a given size turns out to be too
large. Indeed, if we let $\OMCut_L=\{\Gamma\in\OMCut \colon |\Gamma|
= L\}$ then $|\OMCut_L| \geq 2^{\left(\frac{1}{2d}+c_d\right)L}$ for
some $c_d > 0$, at least on a subsequence of $L$s. This can be seen
by counting those cutsets which approximate closely the boundary of
a large cube with sides parallel to the axes of $\Z^d$ (see
Figure~\ref{fig:cubes}), but we neither prove nor use this fact in
our work.

\begin{figure}[h]
  \centering
  \begin{tabular}{ccc}
    \input{fig-cube-0} & \hspace{0.2in} & \input{fig-cube-1}
  \end{tabular}
  \caption{Odd cutsets approximating the boundary of a cube. On the left, every second boundary vertex can be ``pushed out'' independently of other vertices, yielding $2^{\left(\frac{1}{2d}-o_d(1)\right)L}$ odd cutsets with $L$ edges, as $L\to\infty$. On the right, for every odd cutset obtained in such way, we have
the additional option of independently ``pushing out'' vertices all
of whose $2d-2$ neighbors were ``pushed out'' in the first stage.
Combining these two stages shows that as $L\to\infty$ (along a
subsequence) there are at least
$2^{\left(\frac{1}{2d}+c_d-o_d(1)\right)L}$ odd cutsets with $L$
edges for some $c_d>0$.}
  \label{fig:cubes}
\end{figure}
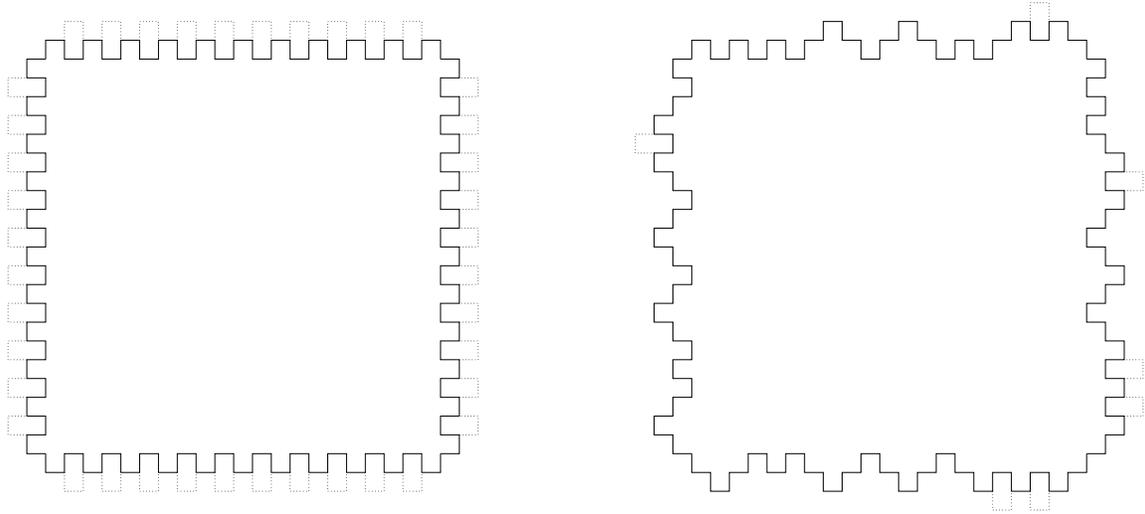

Instead, we introduce a (very coarse) measure of regularity on $\OMCut$ (the quantity $R_\Gamma(E_1(\Gamma))$ defined in Section~\ref{sec:counting-odd-cutsets}) and partition all cutsets into "regular" and "irregular" ones. Theorem~\ref{thm:OMCutMR} shows that the set of irregular cutsets is fairly small and hence the union bound (with the above estimate~\eqref{ineq:Break_is_Gamma}) is sufficient to bound the probability that $\Break(\omega)$ is irregular (see Theorem~\ref{thm:Omega1}).

In order to bound the probability that $\Break(\omega)$ is regular, we partition the set of all regular cutsets into a (relatively) small number of classes and give an estimate on the probability of $\Break(\omega)$ belonging to each such class that is strong enough to facilitate a union bound over all classes. We obtain this partition by exploiting a certain concentration of shape phenomenon. It is shown in Theorems~\ref{thm:OMCut-approx} and~\ref{thm:OMCuteps-approx} that, informally, for sufficiently large $d$ there is a set of "shapes" which contain a good "approximation" to every cutset in $\OMCut_L$, with the number of such shapes significantly smaller than $|\OMCut_L|$. Adapting the shift transformation to this notion of approximation (see Section~\ref{sec:T2}), we are able to bound the probability that $\Break(\omega)$ belongs to the set of cutsets with a given shape. The bound we get is only strong enough for cutsets which are regular (see Theorems~\ref{thm:Omega2} and \ref{thm:Omega2-new}), necessitating the separate treatment of irregular cutsets given above.

\section{The argument}

\subsection{Locating the cutset}

\label{sec:locating_cutset}

Let $d, n \geq 2$ and let $x$ be an arbitrary even vertex of $G_n$. Let $\Fo_n$ be the set of all feasible configurations with $\omega(v) = 1$ for all $v \in B_n \cap \Vo$ and let $\Omega_n$ denote the ``bad'' event, i.e., the set of all $\omega \in \Fo_n$ with $\omega(x) = 1$. For the sake of clarity of the presentation, most of the time we will drop the subscript $n$ from $G_n$, $B_n$, $\muno$, $\Fo_n$, and $\Omega_n$.

Fix an arbitrary ``bad'' configuration $\omega \in \Omega$. Observe that all odd boundary vertices are occupied and all even boundary vertices are vacant, whereas $x$ is an even vertex that is occupied and all its neighbors are odd and vacant. We are going to associate with $\omega$ an edge cutset in $G$ that will mark the outermost break in the above boundary pattern, i.e., the outermost contour separating vertices that are odd and occupied from vertices that are even and occupied. More precisely, we let $X$ be the set of all vacant odd vertices, i.e., $X = \Vo \cap \Vva$, let $A_0'$ be the connected component of $B$ in the graph $G \setminus X$ (note that $B$ is connected in $G$ and disjoint from $X$), and note that $x \not\in A_0'$ as $N(x) \subseteq X$. Finally, let $A_1$ be the connected component of $x$ in the graph $G \setminus A_0'$ and let $\Gamma$ be the set of all edges with one endpoint in $A_0'$ and one endpoint in $A_1$. By definition, every path from $x$ to $B$ must use an edge of $\Gamma$ and no strict subset $\Gamma' \subseteq \Gamma$ has this property. In other words, we may say that $\Gamma$ is a {\em minimal edge cutset} separating $x$ from $B$. To summarize, we have defined a mapping $\Break \colon \Omega \to \cP(E(G))$ that assigns to each configuration in $\Omega$ an edge cutset separating $x$ from $B$. Below, we establish some crucial properties of this cutset.

Let $\Gamma = \Break(\omega)$. For an arbitrary vertex $v \in V(G)$, let $\PG(v)$ be the number of edges of $\Gamma$ that are incident to $v$ and let
\[
E_0 = \{v \in A_0' \colon \PG(v) > 0\} = \{v \not\in A_1 \colon \PG(v) > 0\} \quad \text{and} \quad E_1 = \{v \in A_1 \colon \PG(v) > 0\}.
\]

\begin{prop}
  \label{prop:Gamma}
  $E_0 \subseteq \Ve \cap \Vva$ and $E_1 \subseteq \Vo \cap \Vva$.
\end{prop}
\begin{proof}
  Fix some $v \in E_1$ and let $w$ be an arbitrary vertex with $\{v, w\} \in \Gamma$; at least one such vertex exists since $\PG(v) > 0$. Clearly, $w \in A_0'$ and hence $v \in X$ or otherwise $v$ would also belong to $A_0'$ (recall that $A_0'$ is a maximal connected subset of vertices in $G \setminus X$). It follows that $E_1 \subseteq X = \Vo \cap \Vva$. Since each vertex in $E_0$ is adjacent to a vertex in $E_1$, it immediately follows that $E_0 \subseteq \Ve$. Finally, fix an arbitrary $w \in E_0$. It remains to be shown that $w$ is vacant. Since $w \in A_0'$, then there is a path in $A_0'$ connecting $w$ to $B \cap \Vo$. The immediate neighbor of $w$ on this path is an odd vertex that does not belong to $X$ and hence it is occupied. Therefore, $w$ must be vacant.
\end{proof}

Before we proceed, we have to take a little detour and introduce some terminology that will help us deal with cutsets arising in the procedure described above. Following~\cite{Pe}, we let $\MCut$ be the set of all {\em minimal edge cutsets} separating $x$ and $B$, i.e., the set of all $\Gamma \subseteq E(G)$ such that any path from $x$ to $B$ must cross an edge of $\Gamma$ and no strict subset $\Gamma' \subseteq \Gamma$ shares this property. For every $\Gamma \in \MCut$, let $A_0(\Gamma)$ and $A_1(\Gamma)$ denote the connected components of $B$ and $x$ in $G \setminus \Gamma$, respectively. By minimality of $\Gamma$, every edge of $\Gamma$ must have an endpoint in $A_0(\Gamma)$ and an endpoint in $A_1(\Gamma)$. It follows that $A_0(\Gamma)$ and $A_1(\Gamma)$ form a partition of the vertex set of $G$. For every vertex $v \in V(G)$, let $\PG(v)$ be the number of edges in $\Gamma$ that are incident to $v$, and let
\[
E_0(\Gamma) = \{v \in A_0(\Gamma) \colon \PG(v) > 0\} \quad \text{and} \quad E_1(\Gamma) = \{v \in A_1(\Gamma) \colon \PG(v) > 0\}.
\]
Finally, we define $\OMCut$, the set of {\em odd minimal edge cutsets} (or simply {\em odd cutsets}) to be the set of all $\Gamma \in \MCut$ that satisfy $E_1(\Gamma) \subseteq \Vo$.

Since the new definitions of $A_1$, $E_0$, and $E_1$ coincide with the old ones (more precisely, if $\Gamma = \Break(\omega)$ for some $\omega \in \Omega$ and the procedure described at the beginning of this section defines sets $A_1$, $E_0$, and $E_1$, then $A_1 = A_1(\Gamma)$, $E_0 = E_0(\Gamma)$, and $E_1 = E_1(\Gamma)$), then it is easy to see that the mapping $\Break$ assigns to each configuration in $\Omega$ an odd cutset $\Gamma = \Break(\omega) \in \OMCut$ satisfying Proposition~\ref{prop:Gamma}.

We need to establish one more crucial property of the mapping $\Break$. To this end, with $\Gamma \in \OMCut$ fixed, we say that a configuration $\omega'$ is an {\em interior modification} of another configuration $\omega$ if they agree everywhere but at most on the set $A_1(\Gamma) \setminus E_1(\Gamma)$, i.e., if $\omega(v) = \omega'(v)$ for all $v \not\in A_1(\Gamma) \setminus E_1(\Gamma)$. A moment of thought assures us that the following is true about the map $\Break$.

\begin{prop}
  \label{prop:interior-mod}
  Let $\Gamma \in \OMCut$ and assume that $\Gamma = \Break(\omega)$ for some configuration $\omega \in \Omega$. Then $\Gamma = \Break(\omega')$ for every $\omega'$ that is an interior modification of $\omega$.
\end{prop}
\begin{proof}
  Let $X = \{v \in \Vo \colon \omega(v) = 0\}$ and let $X' = \{v \in \Vo \colon \omega'(v) = 0\}$. Since $\omega'$ is an interior modification of $\omega$, it follows that $X' \setminus A_1(\Gamma) = X \setminus A_1(\Gamma)$ and Proposition~\ref{prop:Gamma} implies that $X, X' \supseteq E_1(\Gamma)$. By definition, every path from $B$ to $A_1(\Gamma)$ has a vertex in $E_1(\Gamma)$. It follows that the connected components of $B$ in $G \setminus X$ and $G \setminus X'$ are identical and hence $\Break(\omega) = \Break(\omega')$. To see this, recall that $\Break(\omega)$ depends solely on the connected component of $B$ in the graph $G \setminus \{v \in \Vo \colon \omega(v) = 0\}$, which we denoted by $A'_0$ at the beginning of this section.
\end{proof}

\subsection{Minimal edge cutsets}

In this section, we give basic definitions and properties of minimal
edge cutsets that will be used throughout the paper. The following
property of the sets $E_0(\Gamma)$ and $E_1(\Gamma)$, which is a
direct consequence of the results proved by Tim{\'a}r~\cite{Ti},
will be crucial in our considerations.

\begin{prop}
  \label{prop:Edelta-connected}
  For every $\Gamma \in \MCut$, the sets $E_0(\Gamma)$ and $E_1(\Gamma)$ are connected in the graph $G^2$.
\end{prop}
\begin{proof}
  Following~\cite{Ti}, for a graph $H$, any $C \subseteq V(H)$, and a $v \in V(H)$, define the outer boundary of $C$ visible from $v$, denoted $\Bvis_H(v, C)$, to be the set of all $y \in N_H(C)$ that are connected to $v$ in the graph $H \setminus C$. It follows from~\cite[Theorem~2]{Ti} that for any connected subset $C \subseteq V(G)$ such that $B \subseteq C$ or $B \cap C = \emptyset$ and any $v \in V(G) \setminus C$, the set $\Bvis_G(v, C)$ is connected in $G^2$. The reason for the assumption that $C$ either contains $B$ or is disjoint from $B$ is that such set $C$, when viewed as a subset in the graph $\Z^d$, satisfies $\Bvis_G(v, C) = \Bvis_{\Z^d}(v, C)$ for every $v \in V(G) \setminus C \subseteq \Z^d$. To conclude, we simply note that the sets $A_0(\Gamma)$ and $A_1(\Gamma)$ are connected, $B \subseteq A_0(\Gamma)$, $B \cap A_1(\Gamma) = \emptyset$, $E_1(\Gamma) = \Bvis_G(x, A_0(\Gamma))$, and $E_0(\Gamma) = \Bvis_G(b, A_1(\Gamma))$ for every $b \in B$.
\end{proof}

For every $j \in [2d]$, we let
\[
E_{1,j}(\Gamma) = \{v \in E_1(\Gamma) \colon \{v, v+f_j\} \in \Gamma \}.
\]
We also define
\[
E_{1,e}(\Gamma) = \{v \in E_1(\Gamma) \colon P_\Gamma(v) \geq 2d - \sqrt{d}\}
\]
and
\[
E_{1,j,x}(\Gamma) = \{v \in E_{1,e}(\Gamma) \colon v + f_j \in A_1(\Gamma)\} = \{v \in E_{1,e}(\Gamma) \colon \{v, v+f_j\} \not\in \Gamma\}.
\]
The letter e stands for {\em exposed} as vertices in $E_{1,e}(\Gamma)$ are exposed to $\Gamma$ from many directions. The sets $E_{1,j}$, $E_{1,e}$, and $E_{1,j,x}$ will play an important role in our considerations. We also let
\[
\Gamma_r = \{vw \in \Gamma \colon v \in E_1(\Gamma) \setminus E_{1,e}(\Gamma) \},
\]
for every $j \in [2d]$, let
\[
\Gamma^j = \{vw \in \Gamma \colon v \in E_1(\Gamma) \text{ and } w = v + f_j\} \quad \text{and} \quad \Gamma_r^j = \Gamma_r \cap \Gamma^j.
\]
Additionally, note that for each $j \in [2d]$, we have that $|E_{1,j}| = |\Gamma^j|$, and $|E_{1,j} \setminus E_{1,e}| = |\Gamma_r^j|$. We will repeatedly use the following simple facts.

\begin{prop}
  \label{prop:NAdelta}
  Let $\Gamma \in \OMCut$, let $\delta \in \{0, 1\}$, and let $v \in E_\delta(\Gamma)$. For every $j \in [2d]$, if $v + f_j \in A_\delta(\Gamma)$ (or, equivalently, $\{v, v+f_j\} \not\in \Gamma$), then $N(v + f_j) \subseteq A_\delta(\Gamma)$.
\end{prop}
\begin{proof}
  Assume first that $\delta = 1$. Let $v \in E_1(\Gamma)$ and let $j \in [2d]$ be such that $v + f_j \in A_1(\Gamma)$. Since $E_1(\Gamma) \subseteq \Vo$, then $v + f_j \in \Ve$. Since $N(w) \subseteq A_1(\Gamma) \cup E_0(\Gamma)$ for every $w \in A_1(\Gamma)$, $N(v + f_j) \subseteq \Vo$, and $E_0(\Gamma) \subseteq \Ve$, it follows that $N(v + f_j) \subseteq A_1(\Gamma)$. The case $\delta = 0$ follows similarly.
\end{proof}

\begin{prop}
  \label{prop:Gamma-size}
  For every $\Gamma \in \OMCut$, $|\Gamma| \geq 2d(2d-1)$.
\end{prop}
\begin{proof}
  For every $j \in [2d]$ and $i \in [2d] \setminus \{j\}$, consider the path $x, x - f_j, x - f_j + f_i, \ldots, x-f_j+\ell f_i$, where $\ell$ is the smallest integer such that $x - f_j + \ell f_i \in B$. Since $\Gamma$ separates $x$ from $B$, at least one edge on each such path must belong to $\Gamma$ and this edge is certainly not $\{x, x - f_j\}$ as $x \in \Ve \cap A_1(\Gamma) \subseteq A_1(\Gamma) \setminus E_1(\Gamma)$ and hence both $x$ and $x - f_j$ belong to $A_1(\Gamma)$. Moreover, the edges of the form $\{x, x-f_j\}$ are the only edges belonging to more than one of these $2d(2d-1)$ paths. It follows that $|\Gamma| \geq 2d(2d-1)$.
\end{proof}

\begin{prop}
  \label{prop:Gamma-non-trivial}
  For every $\Gamma \in \OMCut$ and $v \in V(G)$, $\PG(v) \leq 2d-1$.
\end{prop}
\begin{proof}
  Suppose that $\PG(v) = 2d$ for some $v \in V(G)$. Assume first that $v \in A_1(\Gamma)$. Since $A_1(\Gamma)$ is connected in $G \setminus \Gamma$, $x \in A_1(\Gamma)$, and all $2d$ edges incident to $v$ belong to $\Gamma $, it must be that $A_1(\Gamma) = E_1(\Gamma) = \{v\} = \{x\}$. But recall that $x \in \Ve$, which contradicts the fact that $E_1(\Gamma) \subseteq \Vo$ ($\Gamma$ is an odd cutset). Finally, as $B \subseteq A_0(\Gamma)$ and $A_0(\Gamma)$ is connected, one can quickly rule out the other possibility that $v \in A_0(\Gamma)$ and all $2d$ edges incident to $v$ are in $\Gamma$.
\end{proof}

\subsection{Counting odd cutsets}

\label{sec:counting-odd-cutsets}

In this section, which borrows heavily from~\cite{Pe}, we shall state three theorems that estimate the number of odd cutsets in various settings. Before we do that, we need to introduce some more notation. Following~\cite{Pe}, for every $\Gamma \in \OMCut$, $v \in V(G)$, and $E \subseteq V(G)$, we define
\begin{align*}
  \RG(v) & = \min \{ \PG(v), 2d - \PG(v) \} \quad \text{and} \\
  \RG(E) & = \sum_{v \in E} \RG(v).
\end{align*}
For integers $M$ and $R$, we let
\[
\OMCut(M,R) = \{ \Gamma \in \OMCut \colon |E_1(\Gamma)| = M \text{ and } \RG(E_1(\Gamma)) = R \}.
\]
A key observation that may elucidate the above definitions is that,
since
\begin{equation*}
{2d \choose k} = {2d \choose 2d-k} \leq
(2d)^{\min\{k,2d-k\}}\quad\text{for every $k$ with $0 \leq k \leq
2d$,}
\end{equation*}
then for every $v \in V(G)$, there are at most $(2d)^{\RG(v)}$ ways to choose $\PG(v)$ out of the $2d$ edges incident to $v$. One might think of the parameter $\RG(E_1(\Gamma))$ as a measure of the regularity of $\Gamma$. Note that $\RG(E_1(\Gamma)) \leq |\Gamma|$ and that a value of $\RG(E_1(\Gamma))$ significantly smaller than $|\Gamma|$ indicates that most vertices $v$ in $E_1(\Gamma)$ have $\PG(v)$ close to $2d$; this can be interpreted as some roughness of $\Gamma$. The following result is a straightforward corollary of~\cite[Theorem~4.5]{Pe}.

\begin{thm}
  \label{thm:OMCutMR}
  There exist constants $C$ and $d_0$ such that for all integers $M$, $R$, and $d$ with $d \geq d_0$,
  \[
  |\OMCut(M,R)| \leq \exp\left( \frac{C(\log d)^2}{d} R \right).
  \]
\end{thm}

Recalling the definition of $E_{1,e}(\Gamma)$, we say that a set $E \subseteq V(G)$ is an {\em interior approximation} to $\Gamma$ if
\begin{equation}
  \label{dfn:int_approx}
  E_1(\Gamma) \setminus E_{1,e}(\Gamma) \subseteq E \subseteq A_1(\Gamma).
\end{equation}
The following result is a straightforward corollary of~\cite[Theorem~4.13]{Pe}.

\begin{thm}
  \label{thm:OMCut-approx}
  There exist constants $C$ and $d_0$ such that for all integers $L$ and $d$ with $d \geq d_0$, there exists a family $\E$ of subsets of $V(G)$ satisfying
  \[
  |\E| \leq \exp\left( \frac{C(\log d)^2}{d^{3/2}} L \right)
  \]
  and such that for every $\Gamma \in \OMCut$ with $|\Gamma| = L$, there is an $E \in \E$ that is an interior approximation of $\Gamma$.
\end{thm}

\begin{remark}
  Theorems~ \ref{thm:OMCutMR} and \ref{thm:OMCut-approx} were proved in~\cite{Pe} for an alternative definition of $\OMCut$. There, $\OMCut$ was defined as the set of odd minimal cutsets separating two points (or a set and a point) in a discrete \emph{torus}. However, the theorems apply to our setting, since we can think that $G_n$ is embedded naturally in the discrete torus $\Z_{4n}^d$ and note that then $\OMCut$ (in our definition) is a subset of the set of odd minimal cutsets separating $x$ from a ``far away'' point in this
  torus. In addition, in \cite{Pe}, the upper bound on $|\E|$ in
  Theorem~\ref{thm:OMCut-approx} had an additional factor of 2. In
  our application, this factor can be absorbed in the constant $C$ by
  Proposition~\ref{prop:Gamma-size}.
\end{remark}

One of the main ingredients in our proof of Theorem~\ref{thm:main} that will allow us to improve the bound on $\lambda(d)$ given by Theorem~\ref{thm:GaKa-weak} is a refined version of Theorem~\ref{thm:OMCut-approx}. The main idea behind this improved interior approximation theorem is specializing the family $\E$ from the statement of Theorem~\ref{thm:OMCut-approx} to work only for odd cutsets with a particular distribution of edges adjacent to exposed and non-exposed vertices (i.e., the vertices in $E_{1,e}$ and the vertices in $E_1 \setminus E_{1,e}$). To be more precise, for any $\varepsilon$ with $d^{-1/3} \leq \varepsilon \leq 1/2$, recalling the definition of $\Gamma_r$, we let
\[
\OMCut(\varepsilon) = \{ \Gamma \in \OMCut \colon \varepsilon|\Gamma| < |\Gamma_r| \leq 2\varepsilon|\Gamma| \}.
\]
In other words, $\OMCut(\varepsilon)$ consists of those odd cutsets
whose only about $\varepsilon$-fraction of edges are adjacent to
non-exposed vertices. Since interior approximations ``detect'' only
non-exposed vertices, it should not come at a surprise that as
$\varepsilon$ gets smaller, approximating cutsets in
$\OMCut(\varepsilon)$ becomes easier. We show that the following
statement is true.

\begin{thm}
  \label{thm:OMCuteps-approx}
  There exist constants $C$ and $d_0$ such that for all integers $L$ and $d$ with $d \geq d_0$, and every $\varepsilon$ with $d^{-1/2} \leq \varepsilon \leq 1/2$, there exists a family $\E$ of subsets of $V(G)$ satisfying
  \[
  |\E| \leq \exp\left( \frac{C\sqrt{\varepsilon}(\log d)^2}{d^{3/2}} L \right)
  \]
  and such that for every $\Gamma \in \OMCut(\varepsilon)$ with $|\Gamma| = L$, there is an $E \in \E$ that is an interior approximation to $\Gamma$.
\end{thm}

\begin{remark}
  In fact, our proof of Theorem~\ref{thm:OMCuteps-approx} yields a somewhat stronger property of the family $\E$. We show that for every $\Gamma \in \OMCut(\varepsilon)$ with $|\Gamma| = L$, there is an $E \in \E$ such that $\{v \in E_1(\Gamma) \colon \PG(v) < 2d - \sqrt{\varepsilon d}\} \subseteq E \subseteq A_1(\Gamma)$, see Proposition~\ref{prop:int-approx}.
\end{remark}

As we remarked in the introduction, the proof of our main result, Theorem~\ref{thm:main}, has been split up into the geometric part, which was presented in this section, and the probabilistic part, which we will present in the next few sections. Since these parts are almost completely independent, we will postpone the proof of Theorem~\ref{thm:OMCuteps-approx} to Section~\ref{sec:OMCut-proof} and use it as a ``black box'' when we derive the main result.

\subsection{Definition of the transformations}

In this section, we define the mapping that we alluded to in Section~\ref{sec:outline} and establish its key properties. Recall the definitions of $\Omega$ and $\Fo$ from Section~\ref{sec:locating_cutset}. Throughout this section, we fix some $\omega \in \Omega$, let $\Gamma = \Break(\omega)$, $A_1 = A_1(\Gamma)$, $E_0 = E_0(\Gamma)$, $E_1 = E_1(\Gamma)$, $E_{1,e} = E_{1,e}(\Gamma)$, and $E_{1,j} = E_{1,j}(\Gamma)$ and $E_{1,j,x} = E_{1,j,x}(\Gamma)$ for every $j \in [2d]$. Our transformation will take one of two possible forms, depending on the shape of $\Gamma$.

As a preparatory step, for every $j \in [2d]$, we define the $j$th
{\em shift transformation} $\Shift_j \colon \Omega \to \Fo$ (see
Figure~\ref{fig:shift}) by
\[
\Shift_j(\omega)(v) =
\begin{cases}
  \omega(v + f_j) & \text{if $v \in A_1$}, \\
  \omega(v) & \text{otherwise}.
\end{cases}
\]
We remark that such a transformation already appeared in~\cite{GaKa}.

\begin{prop}
  \label{prop:shift-feasible}
  The configuration $\Shift_j(\omega)$ is indeed feasible with the odd boundary vertices occupied. In other words, $\Shift_j(\omega) \in \Fo$.
\end{prop}
\begin{proof}
  Since $\Gamma \in \OMCut$, then $B \cap A_1 = \emptyset$ and therefore $\Shift_j(\omega)(v) = \omega(v) = 1$ for every $v \in B \cap \Vo$. It remains to check that $\Shift_j(\omega)$ is feasible, i.e., for no $v \in V(G)$ and $i \in [2d]$, both $v$ and $v + f_i$ are occupied. If both $v$ and $v + f_i$ belong to $A_1$ or both $v$ and $v + f_i$ are not in $A_1$, then this follows from the fact that $\omega$ is feasible. Otherwise, assume WLOG that $v \not\in A_1$ and $v + f_i \in A_1$. It follows that $v \in E_0$, so $\Shift_j(\omega)(v) = \omega(v) = 0$ by Proposition~\ref{prop:Gamma}.
\end{proof}

\begin{prop}
  \label{prop:shift-gap}
  For all $v \in E_{1,j}$ and $i \in [2d]$, we have $\Shift_j(\omega)(v + f_i) = 0$.
\end{prop}
\begin{proof}
  Fix some $v \in E_{1,j}$ and $i \in [2d]$. By definition, $v + f_j \in E_0$. If $v + f_i \in E_0$, then $\Shift_j(\omega)(v + f_i) = \omega(v + f_i) = 0$ by Proposition~\ref{prop:Gamma}. If $v + f_i \not\in E_0$, then $v + f_i + f_j \in E_1$ (since $v + f_i + f_j \in A_1$ and it is adjacent to $v + f_j \not\in A_1$). It follows that $\Shift_j(\omega)(v + f_i) = \omega(v + f_i + f_j) = 0$, where the last equality again follows from Proposition~\ref{prop:Gamma}.
\end{proof}

\begin{prop}
  \label{prop:shift-measure}
  The shift transformation is preserves $\muo$, i.e., $\muo(\Shift_j(\omega)) = \muo(\omega)$.
\end{prop}
\begin{proof}
  Since $\Shift_j(\omega)(v) = \omega(v)$ for all $v \not\in A_1$, it suffices to show that $|\{v \in A_1 \colon \Shift_j(\omega)(v) = 1\}| = |\{v \in A_1 \colon \omega(v) = 1\}|$. To see this, note that if $v \in A_1$ and $\Shift_j(\omega)(v) = 1$, then $\omega(v + f_j) = 1$ and hence $v + f_j \in A_1$ since otherwise $v + f_j \in E_0$ and $\omega(v + f_j) = 0$ by Proposition~\ref{prop:Gamma}. Conversely, if $v + f_j \in A_1$ and $\omega(v + f_j) = 1$, then $v + f_j \not\in E_1$ by Proposition~\ref{prop:Gamma} and hence $v \in A_1$ and $\Shift_j(\omega)(v) = \omega(v + f_j) = 1$.
\end{proof}

\subsubsection{The shift transformation}

We are now ready to define the first expanding transformation $T_1 \colon \Omega \to \cP(\Fo)$. First, for every $j \in [2d]$, we define the transformation $T_{1,j} \colon \Omega \to \cP(\Fo)$ by letting $T_{1,j}(\omega)$ be the set of all configurations $\omega'$ of the form
\[
\omega'(v) =
\begin{cases}
  \Shift_j(\omega)(v) & \text{if $v \not\in E_{1,j}$}, \\
  \varepsilon_v & \text{otherwise},
\end{cases}
\]
where $(\varepsilon_v)_v$ is an arbitrary $\{0,1\}$-sequence indexed by $E_{1,j}$. Propositions~\ref{prop:shift-feasible} and~\ref{prop:shift-gap} imply that each such $\omega'$ indeed belongs to $\Fo$ whereas Proposition~\ref{prop:shift-measure} implies that
\[
\muo(T_{1,j}(\omega)) = (1 + \lambda)^{|E_{1,j}|}\muo(\omega).
\]
Next, observe that $|\Gamma| = \sum_{j=1}^{2d} |\Gamma^j|$ and hence $|\Gamma^j| \geq |\Gamma|/(2d)$ for some $j$; in fact, since $\Gamma \in \OMCut$, we have that $|\Gamma^j| = |\Gamma|/(2d)$ for all $j$, see~\cite{Pe}, but we will not need this. We define the transformation $T_1$ by $T_1(\omega) = T_{1, j}(\omega)$, where $j = j(\Gamma) = j(\Break(\omega))$ is the smallest index $j$ satisfying $|\Gamma^j| \geq |\Gamma|/(2d)$. It follows that
\begin{equation}
  \label{ineq:muo-T1-omega}
  \muo(T_1(\omega)) = (1+\lambda)^{|E_{1,j}|}\muo(\omega) = (1+\lambda)^{|\Gamma^j|}\muo(\omega) \geq (1+\lambda)^{|\Gamma|/(2d)}\muo(\omega).
\end{equation}
Finally, we show that we can ``invert'' $T_1$ if we know
$\Break(\omega)$.

\begin{prop}
  \label{prop:T1-recovery}
  For every $\Gamma \in \OMCut$, and $\omega' \in \Fo$, there is at most one $\omega \in \Omega$ satisfying $\Gamma = \Break(\omega)$ and $\omega' \in T_1(\omega)$.
\end{prop}
\begin{proof}
  With $\Gamma \in \OMCut$ fixed, let $A_1 = A_1(\Gamma)$, $E_1 = E_1(\Gamma)$, and $j = j(\Gamma)$. Let $\omega \in \Omega$ satisfy $\Gamma = \Break(\omega)$ and $\omega' \in T_1(\omega) = T_{1,j}(\omega)$. We show that we can recover $\omega$ from $\omega'$. By the definition of $T_{1,j}$, we have $\omega(v) = \omega'(v)$ for all $v \not\in A_1$ and $\omega(v) = \omega'(v - f_j)$ for all $v \in A_1$ such that $v - f_j \in A_1$. Finally, if $v \in A_1$ but $v - f_j \not\in A_1$, then $v \in E_1$ and $\omega(v) = 0$ by Proposition~\ref{prop:Gamma}.
\end{proof}

\subsubsection{The shift+erase transformations}
\label{sec:T2}

Next, we define the second expanding transformation $T_2 \colon \Omega \to \cP(\Fo)$. First, for every $j \in [2d]$ we define the transformation $T_{2,j} \colon \Omega \to \cP(\Fo)$ by letting $T_{2,j}(\omega)$ be the set of all configurations $\omega'$ of the form
\[
\omega'(v) =
\begin{cases}
  \Shift_j(\omega)(v) & \text{if $v \not\in E_{1,j} \cup E_{1,e}$}, \\
  \varepsilon_v & \text{if $v \in E_{1,j} \setminus E_{1,e}$}, \\
  0 & \text{if $v \in E_{1,e}$},
\end{cases}
\]
where $(\varepsilon_v)_v$ is an arbitrary $\{0,1\}$-sequence indexed by $E_{1,j} \setminus E_{1,e}$. Again, Propositions~\ref{prop:shift-feasible} and~\ref{prop:shift-gap} imply that each such $\omega'$ indeed belongs to $\Fo$. With the aim of computing $\muo(T_{2,j}(\omega))$, first let $X_j(\omega)$ denote the set of exposed vertices that are occupied in $\Shift_j(\omega)$, but $T_{2,j}(\omega)$ forces them to be vacant, i.e., $X_j(\omega) = \{v \in E_{1,e} \colon \omega(v + f_j) = 1\}$, and observe that $X_j(\omega) \subseteq E_{1,j,x}$ since if $v \in E_{1,e} \setminus E_{1,j,x}$, then $v + f_j \in E_0$ and hence $\omega(v + f_j) = 0$ by Proposition~\ref{prop:Gamma}. Now it is not hard to see that
\[
\muo(T_{2,j}(\omega)) = (1 + \lambda)^{|E_{1,j} \setminus E_{1,e}|} \lambda^{-|X_j(\omega)|} \muo(\omega).
\]
The transformation $T_2$ is defined by $T_2(\omega) = T_{2,j}(\omega)$, where $j = j(\Gamma) = j(\Break(\omega))$ is the smallest index $j$ that maximizes the quantity $|\Gamma_r^j| - 8|E_{1,j,x}|$. It follows that
\begin{equation}
  \label{ineq:muo-T2-omega}
  \muo(T_2(\omega)) = (1+\lambda)^{|E_{1,j} \setminus E_{1,e}|} \lambda^{-|X_j(\omega)|} \muo(\omega) = (1+\lambda)^{|\Gamma_r^j|} \lambda^{-|X_j(\omega)|} \muo(\omega).
\end{equation}
We close this section by showing how we can ``invert'' $T_2$ if we know an interior approximation $E$ to $\Break(\omega)$ (recall the definition of an interior approximation given in~\eqref{dfn:int_approx}) and the set $X(\omega) = X_{j(\Break(\omega))}(\omega)$. More precisely, we first prove that given $E$ and an $\omega' \in T_2(\omega)$, we can reconstruct $\Break(\omega)$ and then, if we additionally specify the set $X(\omega)$, then $\omega$ is uniquely determined.
\begin{prop}
  \label{prop:T2-recovery}
  For every $E \subseteq V(G)$ and $\omega' \in \Fo$, there is at most one $\Gamma \in \OMCut$ such that the following holds:
  \begin{enumerate}[(i)]
  \item
    \label{item:T2-recovery-i}
    If $\omega' \in T_2(\omega)$ for some $\omega \in \Omega$ such that $E$ is an interior approximation to $Break(\omega)$, then $\Break(\omega) = \Gamma$.
  \item
    \label{item:T2-recovery-ii}
    For every $X \subseteq V(G)$, there is at most one $\omega \in \Omega$ such that $\omega' \in T_2(\omega)$, $E$ is an interior approximation to $\Break(\omega)$, and $X_{j(\Gamma)}(\omega) = X$.
  \end{enumerate}
\end{prop}
\begin{proof}
  Let $\omega \in \Omega$ be any configuration such that $\omega' \in T_2(\omega)$. We first show (i), i.e., that we can recover $\Break(\omega)$ if we know that $E$ is an interior approximation to it. To see this, define a configuration $\omega''$ by
  \[
  \omega''(v) =
  \begin{cases}
    \omega'(v) & \text{if $v \not\in E$}, \\
    0 & \text{if $v \in E$},
  \end{cases}
  \]
  and note that, letting $A_0 = A_0(\Break(\omega))$, $A_1 = A_1(\Break(\omega))$, $E_1 = E_1(\Break(\omega))$, and $E_{1,e} = E_{1,e}(\Break(\omega))$,
  \begin{equation}
    \label{eq:omega''omega}
    \omega''|_{A_0} = \omega'|_{A_0} = \omega|_{A_0} \quad \text{and} \quad \omega''|_{E_1} = 0 = \omega|_{E_1}.
  \end{equation}
  The first identity in~\eqref{eq:omega''omega} follows from the fact that $E \subseteq A_1$; the second identity in~\eqref{eq:omega''omega} follows since $\omega'(v) = 0$ for every $v \in E_{1,e}$ by the definition of $T_2(\omega)$ and since $E_1 \setminus E_{1,e} \subseteq E$. The identities~\eqref{eq:omega''omega} imply that $\omega''$ is an interior modification of $\omega$ (recall the definition of an interior modification given in Section~\ref{sec:locating_cutset}) and hence $\Break(\omega) = \Break(\omega'')$ by Proposition~\ref{prop:interior-mod}.

  In order to see (ii), assume that $\omega' \in T_2(\omega)$ for some $\omega \in \Omega$ such that $E$ is an interior approximation to $\Break(\omega)$ and $X(\omega) = X$. We show that we can recover $\omega$ from $\omega'$, $E$, and $X$. By (i), there is a unique $\Gamma \in \OMCut$ such that $\Gamma = \Break(\omega)$. Let $j = j(\Gamma)$ be such that $T_2(\omega) = T_{2,j}(\Omega)$. Furthermore, let $A_1 = A_1(\Gamma)$, $E_1 = E_1(\Gamma)$, $E_{1,e} = E_{1,e}(\Gamma)$, and $E_{1,j,x} = E_{1,j,x}(\Gamma)$. By the definition of $T_{2,j}$, we have $\omega(v) = \omega'(v)$ for all $v \not\in A_1$ and $\omega(v) = \omega'(v - f_j)$ for all $v \in A_1$ such that $v - f_j \in A_1 \setminus E_{1,e}$. If $v \in A_1$ but $v - f_j \not\in A_1$, then $v \in E_1$ and hence $\omega(v) = 0$ by Proposition~\ref{prop:Gamma}. Finally, if $v \in A_1$ and $v - f_j \in E_{1,e}$, then $\omega(v) = 1$ if $v - f_j \in X$ and $\omega(v) = 0$ otherwise.
\end{proof}

\subsection{Proof of Theorems~\ref{thm:GaKa-weak} and~\ref{thm:main}}

As remarked in the outline, we will split the bad event $\Omega$ into two parts, depending on the ``regularity'' of $\Break(\omega)$. A cutset $\Gamma \in \OMCut$ is ``irregular'' if the ratio $|\Gamma_r| / |\Gamma|$ is ``small'', i.e., if a vast majority of the edges of $\Gamma$ are incident to exposed vertices; otherwise, $\Gamma$ is ``regular''. As the precise meaning of ``small'' (and hence the resulting partition) will be different in the proofs of Theorems~\ref{thm:GaKa-weak} and \ref{thm:main}, we will define a family of such partitions in what might first seem to be unnecessary generality. We fix a non-negative real $\beta$ and let
\[
\Omega_1^\beta = \{\omega \in \Omega \colon |\Gamma_r| < 12|\Gamma|/d^\beta\} \quad \text{and} \quad \Omega_2^\beta = \Omega \setminus \Omega_1^\beta,
\]
where in the above definition $\Gamma$ stands for $\Break(\omega)$. In order to prove Theorems~\ref{thm:GaKa-weak} and~\ref{thm:main}, we will find a $\beta$ such that both $\muo(\Omega_1^\beta)$ and $\muo(\Omega_2^\beta)$ are small under the appropriate assumption on $\lambda$. In particular, Theorem~\ref{thm:GaKa-weak} will easily follow from the following two statements when we set $\beta = 1/4$.

\begin{thm}
  \label{thm:Omega1}
  There exist constants $C$ and $d_0$ such that for every $\beta \in [1/10, 1/2]$, if $d \geq d_0$ and $\lambda \geq Cd^{-\beta}(\log d)^2$, then
  \[
  \muo(\Omega_1^\beta) \leq (1+\lambda)^{-d/4}.
  \]
\end{thm}

\begin{thm}
  \label{thm:Omega2}
  There exist constants $C$ and $d_0$ such that for every $\beta \in [0, 2/5]$, if $d \geq d_0$ and $\lambda \geq Cd^{\beta-1/2}(\log d)^2$, then
  \[
  \muo(\Omega_2^\beta) \leq (1+\lambda)^{-d^{1-\beta}/4}.
  \]
\end{thm}

The proofs of Theorems~\ref{thm:Omega1} and~\ref{thm:Omega2} follow in a straightforward manner from the results established so far plus Theorems~\ref{thm:OMCutMR} and~\ref{thm:OMCut-approx}, which come from~\cite{Pe}. In order to prove Theorem~\ref{thm:main}, we will need the refined interior approximation theorem, Theorem~\ref{thm:OMCuteps-approx}, which allows us to improve the bound on $\muo(\Omega_2^\beta)$ given by Theorem~\ref{thm:Omega2} and whose proof we postponed till Section~\ref{sec:OMCut-proof}. More precisely, Theorem~\ref{thm:main} will follow from Theorem~\ref{thm:Omega1} and the following statement (whose proof relies on Theorem~\ref{thm:OMCuteps-approx}) when we set $\beta = 1/3$.

\begin{thm}
  \label{thm:Omega2-new}
  There exist constants $C$ and $d_0$ such that for every $\beta \in [0,1/2]$, if $d \geq d_0$ and $\lambda \geq Cd^{(\beta-1)/2}(\log d)^2$, then
  \[
  \muo(\Omega_2^\beta) \leq (1+\lambda)^{-d^{1-\beta}/48}.
  \]
\end{thm}

\subsubsection{Proof of Theorem~\ref{thm:Omega1}}

We start by further partitioning the event $\Omega_1^\beta$. Recall the definition of $\RG$ from Section~\ref{sec:counting-odd-cutsets} and for integers $M$ and $R$, let
\[
\Omega_{1,M,R}^\beta = \{\omega \in \Omega_1^\beta \colon |E_1(\Gamma)| = M \text{ and } \RG(E_1(\Gamma)) = R\},
\]
where, as usual, $\Gamma = \Break(\omega)$. Since $\omega \in \Omega_1^\beta$, we also have that
\begin{align*}
  R & = \RG(E_1) = \sum_{v \in E_{1,e}} \RG(v) + \sum_{v \in E_1 \setminus E_{1,e}} \RG(v) \leq \sum_{v \in E_{1,e}} (2d - \PG(v)) + \sum_{v \in E_1 \setminus E_{1,e}} \PG(v) \\
  & \leq |E_{1,e}| \cdot \sqrt{d} + |\Gamma_r| \leq \frac{|\Gamma|}{2d - \sqrt{d}} \cdot \sqrt{d} + 12|\Gamma|/d^\beta \leq 13|\Gamma|/d^\beta,
\end{align*}
where the last inequality follows from our assumption that $\beta \leq 1/2$. It follows that $|\Gamma|/(4d) \geq Rd^{\beta-1}/52$ and hence, recalling the properties of $T_1$, namely inequality~(\ref{ineq:muo-T1-omega}), for every $\omega \in \Omega_{1,M,R}^\beta$,
\begin{equation}
  \label{ineq:muo-T1-omega-bMR}
  \muo(T_1(\omega)) \geq (1+\lambda)^{|\Gamma|/(4d) + Rd^{\beta-1}/52}\muo(\omega) \geq (1+\lambda)^{d/2 + Rd^{\beta-1}/52}\muo(\omega),
\end{equation}
where the last inequality follows from the fact that $|\Gamma| \geq 2d^2$, see Proposition~\ref{prop:Gamma-size}. On the other hand, by Proposition~\ref{prop:T1-recovery}, for every $\Gamma \in \OMCut$ and $\omega' \in \Fo$, there is at most one $\omega \in \Omega$ satisfying $\Break(\omega) = \Gamma$ and $\omega' \in T_1(\omega)$. It follows that there is a constant $C$ such that for every $\omega' \in \Fo$,
\begin{equation}
  \label{ineq:repetitionT1}
  |\{\omega \in \Omega_{1,M,R}^\beta \colon \omega' \in T_1(\omega) \}| \leq |\OMCut(M,R)| \leq \exp\left( \frac{C(\log d)^2}{d} R \right),
\end{equation}
where the last inequality follows from Theorem~\ref{thm:OMCutMR}. Inequalities~\eqref{ineq:muo-T1-omega-bMR} and~\eqref{ineq:repetitionT1} and Lemma~\ref{lemma:double-counting} imply that
\[
\muo(\Omega^\beta_{1,M,R}) \leq \exp\left( \frac{C(\log d)^2}{d} R \right) (1+\lambda)^{-d/2 - Rd^{\beta-1}/52}.
\]

Now we are ready to estimate $\muo(\Omega_1^\beta)$. Since $\RG(E_1(\Gamma)) \geq |E_1(\Gamma)| \geq 1$ for every $\Gamma \in \OMCut$ (since $1 \leq \PG(v) \leq 2d-1$ for every $v \in E_1(\Gamma)$ by Proposition~\ref{prop:Gamma-non-trivial}), we have that
\[
\Omega_1^\beta = \bigcup_{R=1}^\infty \bigcup_{M=1}^R \Omega_{1,M,R}^\beta
\]
and hence by the union bound,
\begin{equation}
  \label{ineq:Omega1b-union}
  \muo(\Omega_1^\beta) \leq \sum_{R=1}^\infty \sum_{M=1}^R \muo(\Omega_{1,M,R}^\beta) \leq \sum_{R = 1}^\infty R \cdot \exp\left( \frac{C(\log d)^2}{d} R \right) (1+\lambda)^{-d/2 - Rd^{\beta-1}/52}.
\end{equation}
In order to estimate the right-hand side of~\eqref{ineq:Omega1b-union}, first observe that
\[
(1+\lambda)^{Rd^{\beta-1}/156} \geq \exp\left( \frac{C(\log d)^2}{d} R \right)
\]
provided that $d$ is sufficiently large and $\lambda \geq C'd^{-\beta}(\log d)^2$ for some large positive constant $C'$. It follows that
\begin{equation}
  \label{ineq:Omega1b-union'}
  \muo(\Omega_1^\beta) \leq \sum_{R = 1}^{\infty} R(1+\lambda)^{-d/2 - Rd^{\beta-1}/78} \leq 2(1+\lambda)^{-d/3},
\end{equation}
provided that $d$ is sufficiently large and $\lambda \geq
C'd^{-\beta}(\log d)^2$. To see that the last inequality
in~\eqref{ineq:Omega1b-union'} holds, one might split the above sum
into ranges $R \leq d^2$ and $R > d^2$ and estimate each of the
parts separately. We conclude that if $d$ is sufficiently large and
$\lambda \geq C'd^{-\beta}(\log d)^2$ for some large positive
constant $C'$, then $\muo(\Omega_1^\beta) \leq (1+\lambda)^{-d/4}$.

\subsubsection{Proof of Theorem~\ref{thm:Omega2}}

Fix an $\omega \in \Omega_2^\beta$. As usual, we let $\Gamma = \Break(\omega)$, let $E_1 = E_1(\Gamma)$, $E_{1,e} = E_{1,e}(\Gamma)$, and for each $j \in [2d]$, let $E_{1,j} = E_{1,j}(\Gamma)$ and $E_{1,j,x} = E_{1,j,x}(\Gamma)$. Recall that each exposed vertex is adjacent to at most $\sqrt{d}$ vertices that are in $A_1(\Gamma)$ and hence
\begin{equation}
  \label{eq:sumE1jx}
  \sum_{j=1}^{2d} |E_{1,j,x}| \leq |E_{1,e}| \cdot \sqrt{d} \leq \frac{|\Gamma|}{2d - \sqrt{d}}\cdot\sqrt{d} \leq \frac{|\Gamma|}{\sqrt{d}}.
\end{equation}
Moreover, since $|\Gamma_r| \geq 12|\Gamma|/d^\beta \geq 12|\Gamma|/\sqrt{d}$ by our assumption on $\beta$, then~\eqref{eq:sumE1jx} implies that
\[
\sum_{j=1}^{2d} (|\Gamma_r^j| - 8|E_{1,j,x}|) = |\Gamma_r| - 8\sum_{j=1}^{2d}|E_{1,j,x}| \geq |\Gamma_r|/3 \geq 4|\Gamma|/d^\beta.
\]
Recall from the definition of $T_2$ that $j(\Gamma)$ is the smallest index that maximizes the quantity $|\Gamma_r^j| - 8|E_{1,j,x}|$. It follows that
\begin{equation}
  \label{ineq:maxj}
  |\Gamma_r^{j(\Gamma)}| - 8|E_{1,j(\Gamma),x}| \geq |\Gamma_r|/(6d) \geq 2L/d^{1+\beta}.
\end{equation}
Next, we further partition the event $\Omega_2^\beta$. For integers $L$, $r$, and $x$, let
\[
\Omega_{2,L,r,x}^\beta = \{\omega \in \Omega_2^\beta \colon |\Gamma| = L, |\Gamma_r^{j(\Gamma)}| = r, \text{ and } |X_{j(\Gamma)}(\omega)| = x\},
\]
where $\Gamma = \Break(\omega)$. Now assume that $\omega \in \Omega_{2,L,r,x}^\beta$. Recalling the properties of $T_2$, namely~\eqref{ineq:muo-T2-omega}, note that since $T_2(\omega) = T_{2,j(\Gamma)}(\omega)$, then we have
\[
\muo(T_2(\omega)) \geq (1+\lambda)^{|E_{1,j(\Omega)} \setminus E_{1,e}|} \lambda^{-|X(\omega)|} \muo(\omega) = (1+\lambda)^r \lambda^{-x} \muo(\omega).
\]
Since $X_j(\omega) \subseteq E_{1,j,x}$, then $r \geq 8x + 2L/d^{1+\beta}$ by \eqref{ineq:maxj} and therefore
\begin{equation}
  \label{ineq:muo-T2-omega-bLrx}
  \muo(T_2(\omega)) \geq (1+\lambda)^{r/2+4x+{L/d^{1+\beta}}}\lambda^{-x}\muo(\omega).
\end{equation}
On the other hand, by Proposition~\ref{prop:T2-recovery}~(\ref{item:T2-recovery-i}), for every $E \subseteq V(G)$ and $\omega' \in \Fo$, there is at most one $\Gamma \in \OMCut$ such that $E$ is an interior approximation to $\Gamma$ and $\Gamma = \Break(\omega)$ for every $\omega \in \Omega$ satisfying $\omega' \in T_2(\omega)$. Moreover, by Proposition~\ref{prop:T2-recovery}~(\ref{item:T2-recovery-ii}), if we additionally specify $X \subseteq V(G)$ and require that $X_{j(\Gamma)}(\omega) = X$, then $\omega$ is uniquely determined. Crucially, since $X_j(\omega) \subseteq E_{1,j,x}(\Gamma)$, then we can assume that $X \subseteq E_{1,j,x}(\Gamma)$. Finally, since $|E_{1,j(\Gamma),x}(\Gamma)| \leq r/8$, then Theorem~\ref{thm:OMCut-approx} implies that there is a constant $C$ such that for every $\omega' \in \Fo$,
\begin{equation}
  \label{ineq:repetitionT2}
  |\{\omega \in \Omega_{2,L,r,x}^\beta \colon \omega' \in T_2(\omega) \}| \leq \exp\left( \frac{C(\log d)^2}{d^{3/2}} L \right) {\lfloor r/8 \rfloor \choose x}.
\end{equation}
Inequalities \eqref{ineq:muo-T2-omega-bLrx} and \eqref{ineq:repetitionT2} and Lemma~\ref{lemma:double-counting} imply that
\begin{equation}
  \label{ineq:muo-T2-Omega2Lrx}
  \muo(\Omega_{2,L,r,x}^\beta) \leq \exp\left( \frac{C(\log d)^2}{d^{3/2}} L \right) {\lfloor r/8 \rfloor \choose x}\lambda^x (1+\lambda)^{-r/2-4x-{L/d^{1+\beta}}}.
\end{equation}
Next, note that if $\lambda \geq 1$, then
\[
(1+\lambda)^{-4x}\lambda^x \leq 1 \quad \text{and} \quad {\lfloor r/8 \rfloor \choose x}(1+\lambda)^{-r/2} \leq 2^{r/8} \cdot 2^{-r/2} \leq 1
\]
and if $\lambda < 1$, then (noting that the function $x \mapsto (y/x)^x$ attains its maximum when $x = y/e$)
\[
(1+\lambda)^{-r/2} {\lfloor r/8 \rfloor \choose x}\lambda^x \leq e^{-\lambda r/4} \left( \frac{\lambda er}{8x} \right)^x \leq e^{-\lambda r/4 + \lambda r/8} \leq 1.
\]
Crucially, if $d$ is sufficiently large and $\lambda \geq C'd^{\beta - 1/2}(\log d)^2$ for a large positive constant $C'$, then
\[
\exp\left( \frac{C(\log d)^2}{d^{3/2}} L \right) \leq (1+\lambda)^{L/(2d^{1+\beta})}.
\]
Putting all of the above together, if $d$ is sufficiently large and $\lambda \geq C'd^{\beta - 1/2}(\log d)^2$, then it follows from~\eqref{ineq:muo-T2-Omega2Lrx} that
\begin{equation}
  \label{ineq:muo-T2-Omega2Lrx-final}
  \muo(\Omega^\beta_{2,L,r,x}) \leq (1+\lambda)^{-L/(2d^{1+\beta})}.
\end{equation}

We are now ready to estimate $\muo(\Omega_2^\beta)$. Since for every $\Gamma \in \OMCut$, we have $|\Gamma_r| \leq |\Gamma|$ and $|\Gamma| \geq 2d^2$ by Proposition~\ref{prop:Gamma-size}, then
\[
\Omega_2^\beta = \bigcup_{L=2d^2}^\infty \bigcup_{r=0}^L \bigcup_{x=0}^{r/8} \Omega_{2,L,r,x}^\beta
\]
and hence by the union bound and~\eqref{ineq:muo-T2-Omega2Lrx-final},
\[
\muo(\Omega_2^\beta) \leq \sum_{L = 2d^2}^{\infty} \sum_{r \leq L} \sum_{x \leq r/8} \muo(\Omega_{2,L,r,x}^\beta) \leq \sum_{L = 2d^2}^{\infty} L^2 \cdot (1+\lambda)^{-L/(2d^{1+\beta})}.
\]
We conclude that if $d$ is sufficiently large and $\lambda \geq C'd^{\beta - 1/2}(\log d)^2$ for some large positive constant $C'$, then $\muo(\Omega_2^\beta) \leq (1+\lambda)^{-d^{1-\beta}/4}$.

\subsubsection{Proof of Theorem~\ref{thm:Omega2-new}}

We will closely follow the proof of Theorem~\ref{thm:Omega2}. The main difference is that we will now further partition the set $\Omega_2^\beta$. To this end, recalling the definition of $\OMCut(\varepsilon)$ from Section~\ref{sec:counting-odd-cutsets}, for $k \in \N$, let
\[
\Omega_2^{\beta,k} = \{ \omega \in \Omega_2^\beta \colon \Break(\omega) \in \OMCut(2^{-k}) \}
\]
and note that $\Omega_2^\beta = \bigcup \{ \Omega_2^{\beta,k} \colon 1 \leq k \leq \beta \log_2 d \}$. Next, for integers $L$, $r$, $x$, and $k$, let
\[
\Omega_{2,L,r,x}^{\beta,k} = \{\omega \in \Omega_2^{\beta,k} \colon |\Gamma| = L, |\Gamma_r^j| = r, \text{ and } |X_j(\omega)| = x\},
\]
where $\Gamma = \Break(\omega)$ and $j = j(\Gamma) \in [2d]$ is such that $T_2(\omega) = T_{2,j}(\omega)$. Fix some $k$ and let $\varepsilon = 2^{-k}$. The first crucial observation that enables us to improve our bound on $\lambda$ is that by the definition of $\OMCut(\varepsilon)$, if $\omega \in \Omega_{2,L,r,x}^{\beta,k}$, then \eqref{ineq:maxj} implies that $r \geq 8x + \varepsilon L/(6d)$ and therefore as in~\eqref{ineq:muo-T2-omega-bLrx},
\begin{equation}
  \label{ineq:muoT2-eps}
  \muo(T_2(\omega)) \geq (1+\lambda)^{r/2+4x+{\varepsilon L/(12d)}}\lambda^{-x}\muo(\omega).
\end{equation}
On the other hand, an argument identical to the one explaining~\eqref{ineq:repetitionT2}, but now using Theorem~\ref{thm:OMCuteps-approx} and the fact that $\Break(\omega) \in \OMCut(\varepsilon)$ for every $\omega \in \Omega_2^{\beta,k}$, implies that there is a positive constant $C$ such that for every $\omega' \in \Fo$,
\begin{equation}
  \label{ineq:repetitionT2-eps}
  |\{\omega \in \Omega_{2,L,r,x}^{\beta,k} \colon \omega' \in T_2(\omega) \}| \leq \exp\left( \frac{C\sqrt{\varepsilon}(\log d)^2}{d^{3/2}} L \right) {\lfloor r/8 \rfloor \choose x}.
\end{equation}
Inequalities \eqref{ineq:muoT2-eps} and \eqref{ineq:repetitionT2-eps} and Lemma~\ref{lemma:double-counting} imply that
\[
\muo(\Omega_{2,L,r,x}^{\beta,k}) \leq \exp\left( \frac{C\sqrt{\varepsilon}(\log d)^2}{d^{3/2}} L \right) {\lfloor r/8 \rfloor \choose x}\lambda^x (1+\lambda)^{-r/2-4x-{\varepsilon L/(12d)}}.
\]
Crucially, if $d$ is sufficiently large and $\lambda \geq C'(\log d)^2 / \sqrt{\varepsilon d}$ for a large positive constant $C'$, then
\[
\exp\left( \frac{C\sqrt{\varepsilon}(\log d)^2}{d^{3/2}} L \right) \leq (1+\lambda)^{\varepsilon L/(24d)}
\]
and, as was shown in the proof of Theorem~\ref{thm:Omega2},
\[
{\lfloor r/8 \rfloor \choose x}\lambda^x(1+\lambda)^{-r/2-4x} \leq 1.
\]
It follows that if $d$ is sufficiently large and $\lambda \geq C'(\log d)^2 / \sqrt{\varepsilon d}$, then
\[
\muo(\Omega^{\beta,k}_{2,L,r,x}) \leq (1+\lambda)^{-\varepsilon L/(24d)}.
\]
Since $\varepsilon \geq d^{-\beta}$, a computation similar to the one done in the proof of Theorem~\ref{thm:Omega2} implies that whenever $d$ is sufficiently large and $\lambda \geq C'(\log d)^2 d^{(\beta - 1)/2}$, then
\[
\muo(\Omega_2^\beta) \leq (1+\lambda)^{-d^{1 - \beta}/48}.
\]

\section{Odd cutsets}

\label{sec:OMCut-proof}

Throughout this section, we fix an $\varepsilon \in [d^{-1/2}, 1/2]$ and assume that $d$ is sufficiently large. Since it will be convenient for us to work with a regular graph, we will consider the graph $G_n$ as a subgraph of the infinite graph $\Z^d$. In particular, we will assume that every vertex of $G$ has exactly $2d$ neighbors. As usual, every time we consider a $\Gamma \in \OMCut$, we let $E_0 = E_0(\Gamma)$, $E_1 = E_1(\Gamma)$, $E_{1,e} = E_{1,e}(\Gamma)$, $A_1 = A_1(\Gamma)$, and $A_0 = A_0(\Gamma) = V(G) \setminus A_1$. Given a $\delta \in \{0, 1\}$ and a condition $c \colon [2d] \to \{0, 1\}$, we will also write
\[
E_{\delta, c(\cdot)} = \{ v \in E_\delta \colon c(\PG(v)) = 1 \}.
\]
For example, $E_{1, \cdot \geq 2d-\sqrt{d}} = E_{1,e}$. Finally, for $\delta \in \{0,1\}$ and $v \in E_\delta$, following~\cite{Pe}, we let
\begin{align*}
  U_1(v) & = \{ v' \in E_\delta \colon vu, uv' \in E(G) \text{ for some $u \in A_\delta$} \}, \\
  U_2(v) & = \{ u \in N(v) \cap A_\delta \colon |N(u) \cap E_\delta| < \sqrt{\varepsilon d} \}, \quad \text{and} \\
  U_3(v) & = (N(U_2(v)) \cap E_\delta) \setminus \{v\}.
\end{align*}

We proceed by establishing a few properties of odd cutsets that will be useful in our further considerations.

\begin{prop}
  \label{prop:PG-sum}
  If $vw \in \Gamma$ for some $v, w \in V(G)$, then $\PG(v) + \PG(w) \geq 2d$.
\end{prop}
\begin{proof}
  Assume WLOG~that $v \in A_1$, write $w = v + f_j$, and note that $w \in A_0$. Let $I$ be the set of those indices $i$ such that $v + f_i \in A_1$ and note that $|I| = 2d - \PG(v)$. Since for every $i \in I$, $w$ is adjacent to $v + f_i + f_j$ and $v + f_i + f_j \in A_1$ by Proposition~\ref{prop:NAdelta}, it follows that $\PG(w) \geq |I|$.
\end{proof}

\begin{prop}
  \label{prop:U1v-size}
  For each $\delta \in \{0, 1\}$ and each $v \in E_\delta$, we have that
  \[
  |U_1(v)| \geq \PG(v)(2d - \PG(v)) - \min\{\PG(v), 2d - \PG(v)\}.
  \]
\end{prop}
\begin{proof}
  Let $J \subseteq [2d]$ be the set of those indices $j$ such that $v + f_j \in A_\delta$ and note that $|J| = 2d-\PG(v)$. Then for all $j \in J$ and $i \in [2d] \setminus J$ such that $f_j \neq -f_i$, we have that $v + f_j + f_i \in U_1(v)$ by Proposition~\ref{prop:NAdelta} and the fact that $v + f_i \in A_{1-\delta}$. Finally, the number of such pairs of $i$ and $j$ is at least $|J|(2d-|J|) - \min\{|J|, 2d-|J|\}$.
\end{proof}

\begin{prop}
  \label{prop:PG-U3v}
  For each $\delta \in \{0, 1\}$, each $v \in E_\delta$, and all $w \in U_3(v)$, we have that $\PG(w) < \sqrt{\varepsilon d}$.
\end{prop}
\begin{proof}
  Let $j \in [2d]$ be such that $w + f_j \in U_2(v)$, let $I \subseteq [2d]$ be the set of those indices $i$ such that $w + f_i \in A_{1-\delta}$, and note that $|I| = \PG(w)$. We deduce from Proposition~\ref{prop:NAdelta} that $w + f_j + f_i \in E_\delta \cap N(w + f_j)$ for every $i \in I$, and hence $|I| < \sqrt{\varepsilon d}$ by the definition of $U_2(v)$.
\end{proof}

\subsection{The dominating set proposition}

The $\sqrt{\varepsilon}$-factor improvement of the upper bound on $\log|\E|$ in Theorem~\ref{thm:OMCuteps-approx} as compared to the bound in Theorem~\ref{thm:OMCut-approx} comes solely from the following refined version of the ``dominating set'' proposition~\cite[Proposition~4.15]{Pe}. Proposition~\ref{prop:Edeltat}, which is the driving force behind the $d^{1/12}$-factor improvement of the upper bound on $\lambda(d)$, is one of the main novelties in this paper. Recall the definitions of $U_1$, $U_2$, and $U_3$ from the beginning of this section and the definition of $\RG$ from Section~\ref{sec:counting-odd-cutsets}.

\begin{prop}
  \label{prop:Edeltat}
  There exists a constant $C$ such that for all $\Gamma \in \OMCut(\varepsilon)$, there exist $E_0^t \subseteq E_0(\Gamma)$ and $E_1^t \subseteq E_1(\Gamma)$ satisfying for both $\delta \in \{0,1\}$:
  \begin{enumerate}[(a)]
  \item
    $\RG(E_\delta^t) \leq \frac{C\sqrt{\varepsilon}\log d}{d^{3/2}}|\Gamma|$.
  \item
    If $v \in E_1$ and $|U_1(v)| \geq \sqrt{\varepsilon}d^{3/2}/2$, then $U_1(v) \cap E_1^t \neq \emptyset$.
  \item
    If $v \in E_{\delta, \cdot \geq d}$, then $|N(v) \cap E_{1-\delta} \cap N(E_\delta^t)| \geq \sqrt{\varepsilon d}$.
  \item
    If $v \in E_{\delta, \cdot \leq \sqrt{d}}$ and $|U_2(v)| \geq d/2$, then $U_3(v) \cap N(E_{1 - \delta}^t) \neq \emptyset$.
  \end{enumerate}
\end{prop}
\begin{proof}
  Fix a $\Gamma \in \OMCut(\varepsilon)$ and recall that $\PG(v) \leq 2d-1$ for all $v \in V(G)$ by Proposition~\ref{prop:Gamma-non-trivial}. For each $v \in E_0 \cup E_1$, we let
  \[
  p_v =
  \begin{cases}
    \frac{30\log d}{(2d - \PG(v))} \cdot \frac{1}{\sqrt{\varepsilon d}} & \text{if $v \in E_1 \setminus E_{1,e}$}, \\
    \frac{30\log d}{(2d - \PG(v))} \cdot \sqrt{\frac{\varepsilon}{d}} & \text{otherwise}.
  \end{cases}
  \]
  Since $\varepsilon \geq d^{-1/2}$, if $d$ is sufficiently large, then $p_v \in (0,1]$ for all $v$. Now, for $\delta \in \{0,1\}$, we choose $E_\delta^s \subseteq E_\delta$ randomly by adding each $v \in E_\delta$ to $E_\delta^s$ with probability $p_v$ independently of all other vertices. We first claim that for each $\delta \in \{0,1\}$,
  \begin{equation}
    \label{ineq:ExRG-small}
    \Ex[\RG(E_\delta^s \cap E_{\delta, \cdot < d})] = \sum_{v \in E_{\delta, \cdot < d}} p_v \cdot \RG(v) = \sum_{v \in E_{\delta, \cdot < d}} p_v \cdot \PG(v) \leq \frac{60\sqrt{\varepsilon}\log d}{d^{3/2}}|\Gamma|.
  \end{equation}
  To see that the last inequality, note that because $1 / (2d - \PG(v)) \leq 1/d$ if $\PG(v) < d$, we have
  \[
  \sum_{v \in E_{0,\cdot < d}} p_v \cdot \PG(v) \leq \frac{30\sqrt{\varepsilon}\log d}{d^{3/2}} \cdot \sum_{v \in E_0} \PG(v) = \frac{30\sqrt{\varepsilon}\log d}{d^{3/2}} |\Gamma|
  \]
  whereas since $\Gamma \in \OMCut(\varepsilon)$, then $\sum_{v \in E_1 \setminus E_{1,e}} \PG(v) \leq 2\varepsilon|\Gamma|$ and hence
  \[
  \sum_{v \in E_{1,\cdot < d}} p_v \cdot \PG(v) \leq \frac{30\log d}{\sqrt{\varepsilon}d^{3/2}} \cdot \sum_{v \in E_1 \setminus E_{1,e}} \PG(v) \leq \frac{30\log d}{\sqrt{\varepsilon}d^{3/2}} \cdot 2\varepsilon|\Gamma|.
  \]
  Moreover,
  \begin{equation}
    \label{ineq:ExRG-large}
    \Ex[\RG(E_\delta^s \cap E_{\delta, \cdot \geq d})]  = \sum_{v \in E_{\delta, \cdot \geq d}} p_v \cdot \RG(v) = \sum_{v \in E_{\delta, \cdot \geq d}} p_v \cdot (2d - \PG(v)) \leq \frac{90\sqrt{\varepsilon}\log d}{d^{3/2}} |\Gamma|,
  \end{equation}
  where the last inequality follows because
  \[
  \sum_{v \in E_{0, \cdot \geq d}} p_v \cdot (2d - \PG(v)) = |E_{0, \cdot \geq d}| \cdot \frac{30\sqrt{\varepsilon}\log d}{\sqrt{d}} \leq \frac{30\sqrt{\varepsilon}\log d}{d^{3/2}}|\Gamma|,
  \]
  whereas since $\Gamma \in \OMCut(\varepsilon)$, then $|E_{1, \cdot \geq d} \setminus E_{1,e}| \leq 2\varepsilon|\Gamma|/d$ and hence
  \[
  \sum_{v \in E_{1, \cdot \geq d}} p_v \cdot (2d - \PG(v)) = \frac{30\log d}{\sqrt{d}}(|E_{1, \cdot \geq d} \setminus E_{1,e}|/\sqrt{\varepsilon} + |E_{1,e}| \cdot \sqrt{\varepsilon}) \leq \frac{30\log d}{d^{3/2}}(2\varepsilon/\sqrt{\varepsilon} + \sqrt{\varepsilon})|\Gamma|.
  \]
  Since $E_\delta = E_{\delta, \cdot < d} \cup E_{\delta, \cdot \geq d}$, it follows that $\RG(E_\delta^s) = \RG(E_\delta^s \cap E_{\delta, \cdot < d}) + \RG(E_\delta^s \cap E_{\delta, \cdot \geq d})$ and hence by~\eqref{ineq:ExRG-small} and \eqref{ineq:ExRG-large},
  \[
  \Ex[\RG(E_\delta^s)] \leq \frac{150\sqrt{\varepsilon}\log d}{d^{3/2}} |\Gamma|.
  \]
  By Markov's inequality,
  \begin{equation}
    \label{ineq:Pa}
    P\left(\RG(E_\delta^s) \geq \frac{450\sqrt{\varepsilon}\log d}{d^{3/2}}|\Gamma| \right) \leq \frac{1}{3}.
  \end{equation}

  \medskip

  Having part (b) in mind, let $v \in E_1$ be such that $|U_1(v)| \geq \sqrt{\varepsilon}d^{3/2}/2$. If $|U_1(v) \cap E_{1,e}| \geq \sqrt{\varepsilon}d^{3/2}/6$, then
  \[
  P(E_1^s \cap U_1(v) = \emptyset) \leq \prod_{w \in U_1(v) \cap E_{1,e}} (1-p_w) \leq \left( 1 - \frac{30\sqrt{\varepsilon}\log d}{\sqrt{d} \cdot \sqrt{d}} \right)^{\sqrt{\varepsilon}d^{3/2}/6} \leq e^{-5\varepsilon\sqrt{d}\log d} \leq \frac{1}{d^5},
  \]
  since $\varepsilon \geq 1/\sqrt{d}$. Otherwise, $|U_1(v) \setminus E_{1,e}| \geq \sqrt{\varepsilon}d^{3/2}/3$ and
  \[
  P(E_1^s \cap U_1(v) = \emptyset) \leq \prod_{w \in U_1(v) \setminus E_{1,e}} (1-p_w) \leq \left( 1 - \frac{30\log d}{2d \cdot \sqrt{\varepsilon d}} \right)^{\sqrt{\varepsilon}d^{3/2}/3} \leq e^{-15\log d/3} = \frac{1}{d^5}.
  \]
  Either way,
  \begin{equation}
    \label{ineq:Pb}
    P(E_1^s \cap U_1(v) = \emptyset) \leq d^{-5}.
  \end{equation}

  \medskip

  Fix a $\delta \in \{0, 1\}$ and let $v \in E_{\delta, \cdot \geq d}$. Having part~(c) in mind, we estimate $P(|N(v) \cap E_{1-\delta} \cap N(E_\delta^s)| < \sqrt{\varepsilon d})$. We first let $B(v) = N(v) \cap E_{1 - \delta, \cdot \geq 2}$ and then for each $w \in B(v)$, we let $E(w) = (N(w) \cap E_\delta) \setminus \{v\}$.

  \begin{claim}
    $|B(v)| \geq d-1$ and each $v' \in E_\delta$ belongs to at most $2$ of the $E(w)$s.
  \end{claim}
  \begin{proof}
    To see the first part, recall that $\PG(v) \leq 2d-1$ by Proposition~\ref{prop:Gamma-non-trivial} and hence there exists a $j \in [2d]$ such that $v + f_j \in A_\delta$. By Proposition~\ref{prop:NAdelta}, $v + f_j + f_i \in A_\delta$ for all $i \in [2d]$. Thus each $i \in [2d]$ for which $v + f_i \not\in A_\delta$ and $f_i \neq -f_j$ satisfies $v + f_i \in B(v)$ since $v + f_i$ is adjacent to both $v$ and $v + f_j + f_i$ and hence $\PG(v+f_i) \geq 2$. The second part is a simple consequence of the fact that $|N(w) \cap N(w')| \leq 2$ for every two distinct $w, w' \in V(G)$.
  \end{proof}

  Now, for each $w \in B(v)$, define a random set $E(w)^s$ by taking each $v' \in E(w)$ into $E(w)^s$ with probability $p_{v'}/2$. It follows from the second part of the above claim that
  \begin{equation}
    \label{eq:EwsEw}
    \bigcup_{w \in B(v)} E(w)^s \text{ is stochastically dominated by } E_\delta^s.
  \end{equation}
  Noting that $\PG(w) \geq 2$ for all $w \in B(v)$ by the definition of $B(v)$, $|E(w)| = \PG(w) - 1$ by the definition of $E(w)$, and $2d - \PG(v') \leq \PG(w)$ for all $v' \in E(w)$ by Proposition~\ref{prop:PG-sum}, we obtain that for sufficiently large $d$,
  \[
  P(E(w)^s = \emptyset) \leq \prod_{v' \in E(w)} \left( 1 - \frac{15\sqrt{\varepsilon}\log d}{(2d-\PG(v'))\sqrt{d}} \right) \leq \left( 1 - \frac{15\sqrt{\varepsilon}\log d}{\PG(w)\sqrt{d}} \right)^{\PG(w) - 1} \leq 1 - 15\sqrt{\frac{\varepsilon}{d}}.
  \]
  Finally, letting $Z_v = |\{w \in B(v) \colon E(w)^s \neq \emptyset \}|$, it follows that $Z_v$ stochastically dominates a $\Bin\left(|B(v)|, 15\sqrt{\varepsilon/d}\right)$ random variable. On the other hand, it follows from \eqref{eq:EwsEw} that $Z_v$ is stochastically dominated by $|\{N(v) \cap E_{1 -\delta} \cap N(E_\delta^s)\}|$. By our claim and Chernoff's inequality, we deduce that for sufficiently large $d$,
  \begin{equation}
    \label{ineq:Pc}
    P(|\{N(v) \cap E_{1 -\delta} \cap N(E_\delta^s)\}| < \sqrt{\varepsilon d}) \leq P(Z_v < \sqrt{\varepsilon d}) \leq e^{-\sqrt{\varepsilon d}} \leq d^{-5}.
  \end{equation}

  Fix a $\delta \in \{0, 1\}$ and let $v \in E_{\delta, \cdot \leq \sqrt{d}}$ satisfy $|U_2(v)| \geq d/2$. Having part (d) in mind, for each $w \in U_3(v)$, we let $F(w) = (N(w) \cap E_{1-\delta}) \setminus N(v)$ and let $U_3'(v) = \{w \in U_3(v) \colon F(w) \neq \emptyset\}$.

\begin{claim}
  For every $w \in U_3(v)$, $|F(w)| \geq \PG(w) - 1$ and $|U_3'(v)| \geq |U_3(v)| - \PG(v) \geq d/5$.
\end{claim}
\begin{proof}
  To see the first part, observe that for each $w \in U_3(v)$, since $w \in E_\delta$, then we have $|N(w) \cap E_{1-\delta}| = \PG(w) \geq 1$. Moreover, since $N(w) \cap N(v)$ contains a vertex of $U_2(v) \subseteq A_\delta$ and $|N(w) \cap N(v)| \leq 2$ since $w \neq v$, then $|N(w) \cap E_{1-\delta} \cap N(v)| \leq 1$. To see the second part, note first that since each $w \in U_3(v)$ has at most $2$ common neighbors with $v$ and at least one of them belongs to $U_2(v)$, then $|U_3(v)| - \PG(v) \geq |U_2(v)|/2 - \sqrt{d} \geq d/5$, where the last inequality follows from the assumption that $|U_2(v)| \geq d/2$. Next, recall that $\PG(u) \leq 2d-1$ for every vertex $u$ by Proposition~\ref{prop:Gamma-non-trivial} and let $j \in [2d]$ be such that $v + f_j \in E_{1-\delta}$. By the first part of this claim, it suffices to prove that there is at most one $w \in U_3(v)$ with $N(w) \cap E_{1-\delta} = \{v + f_j\}$. Since $\PG(v + f_j) \leq 2d-1$, there is an $i \in [2d]$ such that $v + f_j + f_i \in A_{1-\delta}$. Assume that some $w \in U_3(v)$ satisfies $N(w) \cap E_{1-\delta} = \{v + f_j\}$. Then clearly, $w = v + f_j + f_k$ for some $k \in [2d]$. If $f_k \neq -f_i$, then $\PG(w) \geq 2$ as $w + f_i \in N(v + f_j + f_i) \subseteq A_{1 - \delta}$ by Proposition~\ref{prop:NAdelta} and $w + f_i \neq v + f_j$. Hence, if $N(w) \cap E_{1-\delta} = \{v + f_j\}$, then $w = v + f_j - f_i$.
\end{proof}

Now, for each $w \in U_3'(v)$, define a random set $F(w)^s$ by independently taking each $v' \in F(w)$ into $F(w)^s$ with probability $p_{v'}/3$. Since each $w \in U_3'(v)$ is in distance $2$ from $v$ and each $F(w)$ consists only of vertices in distance $3$ from $v$, it follows that each $v' \in E_{1 - \delta}$ belongs to at most $3$ different $F(w)$s and hence
\begin{equation}
  \label{eq:FwsFw}
  \bigcup_{w \in U_3'(v)} F(w)^s \text{ is stochastically dominated by } E_{1-\delta}^s.
\end{equation}
Noting that $|F(w)| \geq \max\{\PG(w) - 1, 1\}$ (by the above Claim) and $2d - \PG(v') \leq \PG(w)$ (by Proposition~\ref{prop:PG-sum}) for all $w \in U'_3(v)$ and $v' \in F(w)$, we obtain that for sufficiently large $d$,
\[
P(F(w)^s = \emptyset) \leq \prod_{v' \in F(w)} \left( 1 - \frac{10\sqrt{\varepsilon}\log d}{(2d - \PG(v'))\sqrt{d}} \right) \leq \left( 1 - \frac{10\sqrt{\varepsilon}\log d}{\PG(w)\sqrt{d}} \right)^{|F(w)|} \leq 1 - 10\sqrt{\frac{\varepsilon}{d}}.
\]
Finally, letting $Z_v = |\{ w \in U_3'(v) \colon F(w)^s \neq \emptyset \}|$, it follows from \eqref{eq:FwsFw} that $Z_v$ is stochastically dominated by $|\{ w \in U_3(v) \colon N(v) \cap E_{1-\delta}^s \neq \emptyset \}|$. We deduce that for sufficiently large~$d$, by the above Claim,
\begin{equation}
  \label{ineq:Pd}
  P( U_3(v) \cap N(E_{1-\delta}^s) = \emptyset ) \leq P(Z_v = 0) \leq \left( 1 - 10\sqrt{\varepsilon/d} \right)^{|U_3'(v)|} \leq \left( 1 - 10\sqrt{\varepsilon/d} \right)^{d/5} \leq d^{-5},
\end{equation}
where the last inequality follows from the assumption that $\varepsilon \geq d^{-1/2}$.

We now aim to enlarge the sets $E_\delta^s$ slightly to create new sets $E_\delta^t$ that will satisfy the requirements of the proposition. Defining for each $\delta \in \{0, 1\}$,
\begin{align*}
  E_{1,1}^B & = \{ v \in E_{1} \colon |U_1(v)| \geq \sqrt{\varepsilon}d^{3/2}/2 \text{ and } U_1(v) \cap E_1^s = \emptyset \}, \\
  E_{\delta,2}^B & = \{ v \in E_{\delta, \cdot \geq d} \colon |N(v) \cap E_{1-\delta} \cap N(E_\delta^s)| < \sqrt{\varepsilon d} \}, \\
  E_{\delta,3}^B & = \{ v \in E_{\delta, \cdot \leq \sqrt{d}} \colon |U_2(v)| \geq d/2 \text{ and } U_3(v) \cap N(E_{1-\delta}^s) = \emptyset \}
\end{align*}
and using the three probabilistic estimates \eqref{ineq:Pb}, \eqref{ineq:Pc}, and \eqref{ineq:Pd}, we see that
\begin{equation}
  \label{ineq:maxEB}
  \max\left\{ \Ex[|E_{1,1}^B|], \Ex[|E_{\delta,2}^B|], \Ex[|E_{\delta,3}^B|] \right\} \leq \max\{|E_0|, |E_1|\} / d^5 \leq |\Gamma|/d^5.
\end{equation}
In order to guarantee that parts~(b), (c), and (d) of the proposition will be satisfied, for each $\delta \in \{0, 1\}$, we let $E_{\delta,3}^c$ be an arbitrary set of size at most $|E_{1-\delta,3}^B|$ containing a vertex from $N(U_3(v)) \cap E_\delta$ for every $v \in E_{1-\delta,3}^B$ and let
\[
E_0^t = E_0^s \cup E_{0,2}^B \cup E_{0,3}^c \quad \text{and} \quad E_1^t = E_1^s \cup E_{1,1}^B \cup E_{1,2}^B \cup E_{1,3}^c.
\]
Since $v \in U_1(v)$ for every $v \in E_1$ and $|N(v) \cap E_{1-\delta} \cap N(E_\delta^t)| = \PG(v)$ if $v \in E_\delta^t$, then by definition, $E_0^t$ and $E_1^t$ satisfy parts~(b), (c), and (d) of this proposition. Moreover, if we let $M = \max_{\delta,i}\{|E_{\delta,i}^B|\}$, then for each $\delta \in \{0,1\}$,
\[
\RG(E_\delta^t) \leq \RG(E_\delta^s) + 3dM.
\]
Hence, in order to guarantee that part~(a) holds, it is sufficient to show that with positive probability $\max_{\delta} \RG(E_\delta^s) \leq \left(C\sqrt{\varepsilon}\log d/d^{3/2}\right) \cdot |\Gamma|$ and $M \leq C|\Gamma|/d^5$ for some positive constant $C$. By Markov's inequality and \eqref{ineq:maxEB}, we have that
\[
P\left( M > 15|\Gamma|/d^5 \right) \leq P\left( \sum_{\delta, i} |E_{\delta,i}^B| > 15|\Gamma|/d^5 \right) < 1/3.
\]
Combined with \eqref{ineq:Pa} and the union bound, this completes the proof.
\end{proof}

\begin{prop}
  \label{prop:Et-G8-connected}
  Let $\Gamma \in \OMCut$ and assume that two sets $E_0^t \subseteq E_0$ and $E_1^t \subseteq E_1$ satisfy part~(c) of Proposition~\ref{prop:Edeltat}. Then for every $v \in E_1$, the set $\{v\} \cup E_0^t \cup E_1^t$ is connected in the graph $G^8$.
\end{prop}
\begin{proof}
  Since by Proposition~\ref{prop:Edelta-connected}, the set $E_0 \cup E_1$ is connected in $G^2$ and $\{v\} \cup E_0^t \cup E_1^t \subseteq E_0 \cup E_1$, it clearly suffices to prove that $\dist_G(w, E_0^t \cup E_1^t) \leq 3$ for each $w \in E_0 \cup E_1$. To see this, fix some $\delta \in \{0, 1\}$ and $w \in E_\delta$. If $w \in E_{\delta, \cdot \geq d}$, then the fact that $E_\delta^t$ satisfies part~(c) of Proposition~\ref{prop:Edeltat} implies that $\dist_G(w, E_\delta^t) \leq 2$. If $w \in E_{\delta, \cdot < d}$, then by Proposition~\ref{prop:PG-sum}, we can find a $w' \in N(w) \cap E_{1-\delta, \cdot \geq d}$ and hence $\dist_G(w, E_{1-\delta}^t) \leq 1 + \dist_G(w', E_{1-\delta}^t) \leq 3$, again using the fact that $E_{1-\delta}^t$ satisfies part~(c) of Proposition~\ref{prop:Edeltat}.
\end{proof}

\subsection{Constructing an interior approximation}

\label{sec:constr-int-approx}

For $\Gamma \in \OMCut$, $v \in V(G)$, and $E \subseteq E(G)$, define $\NG(v) \colon [2d] \to \{0, 1\}$ by
\[
\NG(v)_j =
\begin{cases}
  1, & \text{if } v+f_j \in A_1, \\
  0, & \text{if } v+f_j \in A_0,
\end{cases}
\]
and let $\NG(E) = (\NG(v))_{v \in E}$. The next proposition formalizes the fact that for every $\Gamma \in \OMCut$, knowing only $E_0^t$ and $E_1^t$ satisfying parts~(b), (c), and (d) from Proposition~\ref{prop:Edeltat} and $\NG(E_\delta^t)$ for both $\delta \in \{0, 1\}$, we can construct a set $E$ that is an interior approximation to $\Gamma$. Such an $E$ is determined by the following algorithm:
\begin{enumerate}[(1)]
\item
  For $\delta \in \{0, 1\}$, let
  \begin{enumerate}[(a)]
  \item
    $R_\delta^a = \{ v + f_j \colon v \in E_{1-\delta}^t \text{ and } \NG(v)_j = \delta \}$,
  \item
    $R_\delta^b = \bigcup \{ N(v + f_j) \colon v \in E_\delta^t \text{ and } \NG(v)_j = \delta \}$.
  \end{enumerate}
\item
  For $\delta \in \{0, 1\}$, let $V_\delta = \{v \in V(G) \colon |N(v) \cap R_{1-\delta}^a| < \sqrt{\varepsilon d}\}$ and
  \[
  U = \{ u \in V_0 \colon N(u) \cap V_1 \cap R_1^a \neq \emptyset \}.
  \]
  Set $E = R_1^b \cup N(U)$.
\end{enumerate}

\begin{prop}
  \label{prop:int-approx}
  For any $\Gamma \in \OMCut$, the set $E$ obtained from the previous algorithm, taking as input sets $(E_\delta^t)_{\delta \in \{0, 1\}}$ satisfying parts~(b), (c), and (d) of Proposition~\ref{prop:Edeltat} and $(\NG(E_\delta^t))_{\delta \in \{0, 1\}}$ satisfies
  \[
  E_{1, \cdot < 2d - \sqrt{\varepsilon d}} \subseteq E \subseteq A_1.
  \]
  In particular, $E$ is an interior approximation to $\Gamma$.
\end{prop}

To gain some understanding of the above algorithm, note that $R_\delta^a$ and $R_\delta^b$ consist of vertices that we know are in $E_\delta$ and $A_\delta$, respectively, directly from the fact that $E_\delta^t \subseteq E_\delta$ and the definition of $\NG(E_\delta^t)$. It is relatively straightforward to show that $E_{1, \sqrt{\varepsilon d} < \cdot < 2d - \sqrt{\varepsilon d}} \subseteq R_1^b$ and hence our main difficulty lies in showing that vertices of $E_{1, \cdot \leq \sqrt{\varepsilon d}}$ can also be recovered. To this end, we define $V_\delta$, which is shown to be disjoint from $E_{\delta, \cdot \geq d}$. We deduce that $U \subseteq A_1 \cap \Ve$ and hence $N(U) \subseteq A_1$ from the definition of $\OMCut$. Finally, we are able to show that $E_{1, \cdot \leq \sqrt{d}} \setminus R_1^b \subseteq N(U)$.

\begin{proof}[Proof of Proposition~\ref{prop:int-approx}]
  The proof is via several claims.

  \begin{cl}
    \label{cl:1}
    For each $\delta \in \{0, 1\}$, we have $R_\delta^a \subseteq E_\delta$ and $R_\delta^b \subseteq A_\delta$.
  \end{cl}
  By the definition of $R_\delta^a$, each of its elements is of the form $v + f_j$ for some $v \in E_{1-\delta}^t \subseteq E_{1-\delta}$ and $j \in [2d]$ such that $\NG(v)_j = \delta$. Then $v + f_j \in A_\delta$ by the definition of $\NG(v)$, but since $v \in E_{1-\delta}$, then in fact $v + f_j \in E_\delta$. Each element of $R_\delta^b$ belongs to $N(v + f_j)$ for some $v \in E_\delta^t \subseteq E_\delta$ and $j \in [2d]$ such that $\NG(v)_j = \delta$. Since $v \in E_\delta$ and $v + f_j \in A_\delta$, then $N(v + f_j) \subseteq A_\delta$ by Proposition~\ref{prop:NAdelta}.

  \begin{cl}
    \label{cl:2}
    $E_{1, \sqrt{\varepsilon d} < \cdot < 2d - \sqrt{\varepsilon d}} \subseteq R_1^b$.
  \end{cl}
  Fix a $v \in E_{1, \sqrt{\varepsilon d} < \cdot < 2d - \sqrt{\varepsilon d}}$ and note that $|U_1(v)| \geq \sqrt{\varepsilon d}(2d - \sqrt{\varepsilon d}) - \sqrt{\varepsilon d} \geq \sqrt{\varepsilon}d^{3/2}/2$ by Proposition~\ref{prop:U1v-size}. Since $E_1^t$ satisfies part (b) of Proposition~\ref{prop:Edeltat}, it follows that $E_1^t \cap U_1(v) \neq \emptyset$ and therefore $v \in R_1^b$.

  \begin{cl}
    \label{cl:3}
    For each $\delta \in \{0, 1\}$, we have $E_{\delta, \cdot \geq d} \cap V_\delta = \emptyset$.
  \end{cl}
  Fix $\delta \in \{0, 1\}$ and $v \in E_{\delta, \cdot \geq d}$. Any vertex in $N(v) \cap E_{1-\delta} \cap N(E_\delta^t)$ is in $N(v) \cap R_{1-\delta}^a$. The claim follows since $E_\delta^t$ satisfies part~(c) of Proposition~\ref{prop:Edeltat}.

  \begin{cl}
    \label{cl:4}
    $U \subseteq A_1 \cap \Ve$.
  \end{cl}
  Let $u \in U$. Since $N(u) \cap R_1^a \neq \emptyset$ and $R_1^a \subseteq E_1 \subseteq \Vo$ by Claim~\ref{cl:1} and the definition of $\OMCut$, then $u \in N(R_1^a) \subseteq \Ve$. Assume for contradiction that $u \not\in A_1$. Since $R_1^a \subseteq E_1$ and $u \in N(R_1^a)$, then $u \in E_0$. If $\PG(u) \geq d$, then $u \not\in V_0$ by Claim~\ref{cl:3}, contradicting the fact that $u \in U$. If $\PG(u) < d$, then by Proposition~\ref{prop:PG-sum}, $\PG(v) \geq 2d - \PG(u) \geq d$ for every $v \in E_1 \cap N(u)$. It follows form Claim~\ref{cl:3} that $N(u) \cap E_1 \cap V_1 = \emptyset$, so in particular, $N(u) \cap R_1^a \cap V_1 = \emptyset$, which again contradicts the fact that $u \in U$.

  \begin{cl}
    \label{cl:5}
    $N(U) \subseteq A_1$.
  \end{cl}
  This follows immediately from Claim~\ref{cl:4} since $E_1$, the boundary of $A_1$, is a subset of $\Vo$.

  \begin{cl}
    \label{cl:6}
    $E_{1, \cdot \leq \sqrt{d}} \subseteq E$.
  \end{cl}
  Let $v \in E_{1, \cdot \leq \sqrt{d}}$. We consider two cases.

  \smallskip
  \noindent
  {\bf Case 1.} $|U_2(v)| < d/2$. By the definition of $U_2(v)$, for any $j \in [2d]$ such that $v + f_j \in A_1 \setminus U_2(v)$, we have at least $\sqrt{\varepsilon d}$ indices $i \in [2d]$ such that $v + f_j + f_i \in E_1$. Since there are $2d - |U_2(v)| - \PG(v)$ such indices $j$ and each vertex can be represented in the form $v + f_j + f_i$ in at most two ways, we have that $|U_1(v)| \geq (2d - |U_2(v)| - \PG(v))\sqrt{\varepsilon d}/2 \geq \sqrt{\varepsilon}d^{3/2}/2$ provided that $d$ is sufficiently large. Since $E_1^t$ satisfies part (b) of Proposition~\ref{prop:Edeltat}, it follows that $U_1(v) \cap E_1^t \neq \emptyset$ and hence $v \in R_1^b$.

  \smallskip
  \noindent
  {\bf Case 2.} $|U_2(v)| \geq d/2$. In this case, since $E_0^t$ satisfies part (d) of Proposition~\ref{prop:Edeltat}, there exist $i, j \in [2d]$ such that $v + f_j \in U_2(v)$ and $v + f_j + f_i \in U_3(v) \cap N(E_0^t)$. In particular, $v + f_j + f_i \in R_1^a$. Using that $R_1^a \subseteq E_1$ by Claim~\ref{cl:1}, we have that $|N(v + f_j) \cap R_1^a| \leq |N(v + f_j) \cap E_1| < \sqrt{\varepsilon d}$ and hence $v + f_j \in V_0$. Proposition~\ref{prop:PG-U3v} implies that $|N(v + f_j + f_i) \cap E_0| = \PG(v + f_j + f_i) < \sqrt{\varepsilon d}$. Hence, since $R_0^a \subseteq E_0$ by Claim~\ref{cl:1}, we have that $v + f_j + f_i \in V_1$. It follows that $v + f_j \in U$ and hence $v \in N(U)$.

  \medskip
  Proposition~\ref{prop:int-approx} follows from Claims~\ref{cl:1}, \ref{cl:2}, \ref{cl:5}, and \ref{cl:6}.
\end{proof}

\subsection{Counting interior approximations}

Given $R \in \N$ and $E \subseteq V(G)$, define
\[
\cN(R, E) = \{ \NG(E) \colon \Gamma \in \OMCut \text{ satisfying } E \subseteq E_0(\Gamma) \cup E_1(\Gamma) \text{ and } \RG(E) = R \}.
\]
Since $1 \leq \PG(v) \leq 2d-1$ for every $v \in E_0(\Gamma) \cup E_1(\Gamma)$ by Proposition~\ref{prop:Gamma-non-trivial}, then $\RG(E) \geq |E|$ for every $E \subseteq E_0(\Gamma) \cup E_1(\Gamma)$. It follows that if $|E| > R$, then $\cN(R,E) = \emptyset$.

\begin{prop}
  \label{prop:cNRE}
  For every $R \in \N$ and $E \subseteq V(G)$, we have $|\cN(R, E)| \leq (2d)^{2R}$.
\end{prop}
\begin{proof}
  As we remarked above, WLOG~we may assume that $|E| \leq R$. Let
  \[
  \Delta = \left\{ D \in [2d-1]^E \colon \sum_{v \in E} \min \{ D_v, 2d - D_v \} = R \right\}
  \]
  and note that if $\Gamma \in \OMCut$ satisfies $E \subseteq E_0(\Gamma) \cup E_1(\Gamma)$ and $\RG(E) = R$, then $(\PG(v))_{v \in E} \in \Delta$. Moreover, given $v \in E$ and $\PG(v)$, the number of possibilities for $\NG(v)$ is at most ${2d \choose \PG(v)}$. It follows that
  \[
    |\cN(R, E)| \leq \sum_{D \in \Delta} \prod_{v \in E} {2d \choose D_v} \leq \sum_{D \in \Delta} \prod_{v \in E} (2d)^{\min\{D_v, 2d - D_v\}} = |\Delta| \cdot (2d)^R \leq (2d)^{|E|+R} \leq (2d)^{2R}.
    \qedhere
  \]
\end{proof}

\begin{proof}[{\bf Proof of Theorem~\ref{thm:OMCuteps-approx}}]
  Fix an integer $L$, $\varepsilon \in [d^{-1/2},1/2]$, and let $R' = \frac{2C\sqrt{\varepsilon}\log d}{d^{3/2}}L$, where $C$ is the constant from the statement of Proposition~\ref{prop:Edeltat}. For every $\Gamma \in \OMCut(\varepsilon)$ and each $\delta \in \{0, 1\}$, let $E_\delta^t(\Gamma) \subseteq E_\delta(\Gamma)$ be an arbitrary set satisfying parts~(a)--(d) of Proposition~\ref{prop:Edeltat}. Let $\cA$ be the algorithm described in Section~\ref{sec:constr-int-approx}. We show that the family $\E$ defined by
  \[
  \E = \{ \cA(E_0^t(\Gamma), E_1^t(\Gamma), \NG(E_0^t(\Gamma)), \NG(E_1^t(\Gamma))) \colon \Gamma \in \OMCut(\varepsilon) \text{ and } |\Gamma| = L \}
  \]
  satisfies the assertion of Theorem~\ref{thm:OMCuteps-approx}. It follows from Proposition~\ref{prop:int-approx} that $\E$ contains an interior approximation to every $\Gamma \in \OMCut(\varepsilon)$ with $|\Gamma| = L$, so in order to complete the proof, we only need to give an upper bound on the cardinality of $\E$. Recall that $x$ is a fixed even vertex of $G_n$ that is separated from $B_n$ by every cutset in $\OMCut$. Hence, for every $\Gamma$ as above, there is an $\ell_\Gamma \in \{0, \ldots, L\}$ such that $x + \ell_\Gamma f_1 \in E_1(\Gamma)$. In particular, it follows from Proposition~\ref{prop:Et-G8-connected} that $\{x + \ell_\Gamma f_1\} \cup E_0^t(\Gamma) \cup E_1^t(\Gamma)$ is connected in $G^8$. Certainly, $|\E|$ is not larger than the number of tuples $(\ell, R, E_0^t, E_1^t, X)$ such that $0 \leq \ell \leq L$, $0 \leq R \leq R'$, $X \in \cN(R, E_0^t \cup E_1^t)$,
  \begin{equation}
    \label{eq:Edeltat-properties}
    E_0^t, E_1^t \subseteq V(G), \quad |E_0^t|, |E_1^t| \leq R, \quad \text{and} \quad \{x + \ell f_1\} \cup E_0^t \cup E_1^t \text{ is connected in } G^8.
  \end{equation}
  For non-negative $\ell$ and $R$, by Lemma~\ref{lemma:connected-sets}, the number $S_{\ell, R}$ of sets satisfying~\eqref{eq:Edeltat-properties} can be bounded as follows
  \[
  S_{\ell, R} \leq \sum_{M=0}^{2R+1} 2^M \cdot \left(\Delta(G^8)\right)^{2M-2} \leq \sum_{M=0}^{2R+1} 2^M \cdot \left((2d)^9\right)^{2M} \leq \sum_{M = 0}^{2R} (2d)^{19M} \leq (2d)^{38R+1} \leq (2d)^{39R}.
  \]
  Now, with $E_0^t$ and $E_1^t$ satisfying~\eqref{eq:Edeltat-properties} fixed, $|\cN(R, E_0^t \cup E_1^t)| \leq (2d)^{2R}$ by Proposition~\ref{prop:cNRE}. It follows that if $d$ is sufficiently large, then
  \[
  |\E| \leq \sum_{\ell = 0}^L \sum_{R = 0}^{R'} (2d)^{39R + 2R} \leq (L + 1) \cdot \sum_{R = 0}^{R'} (2d)^{41R} \leq (L+1) \cdot (2d)^{41R' + 1} \leq \exp\left(\frac{C'\sqrt{\varepsilon}(\log d)^2}{d^{3/2}}L\right),
  \]
  where the last inequality follows from the fact that $L \geq 2d^2$ (which we may assume by Proposition~\ref{prop:Gamma-size}) and $\varepsilon \geq d^{-1/2}$.
\end{proof}

\bibliographystyle{amsplain}
\nocite{*}
\bibliography{hard-core}

\end{document}

%% file: fig-simulation-41x41-1.0.tex
\begin{tikzpicture}

\foreach \p in {(0.000, 0.175), (0.000, 0.525), (0.000, 0.875), (0.000, 1.225), (0.000, 1.575), (0.000, 1.925), (0.000, 2.275), (0.000, 2.625), (0.000, 2.975), (0.000, 3.325), (0.000, 3.675), (0.000, 4.025), (0.000, 4.375), (0.000, 4.725), (0.000, 5.075), (0.000, 5.425), (0.000, 5.775), (0.000, 6.125), (0.000, 6.475), (0.000, 6.825), (0.175, 0.000), (0.175, 0.350), (0.175, 0.700), (0.175, 1.750), (0.175, 2.100), (0.175, 2.450), (0.175, 3.850), (0.175, 4.200), (0.175, 4.550), (0.175, 4.900), (0.175, 7.000), (0.350, 1.575), (0.350, 2.625), (0.350, 4.725), (0.525, 0.000), (0.525, 1.750), (0.525, 3.500), (0.525, 5.600), (0.525, 5.950), (0.525, 7.000), (0.700, 3.325), (0.700, 4.025), (0.700, 5.775), (0.700, 6.125), (0.875, 0.000), (0.875, 1.400), (0.875, 3.150), (0.875, 3.500), (0.875, 6.650), (0.875, 7.000), (1.050, 0.175), (1.050, 2.625), (1.050, 2.975), (1.050, 4.025), (1.050, 4.375), (1.050, 5.775), (1.050, 6.125), (1.225, 0.000), (1.225, 2.100), (1.225, 2.450), (1.225, 2.800), (1.225, 3.150), (1.225, 3.500), (1.225, 4.550), (1.225, 5.600), (1.225, 5.950), (1.225, 6.300), (1.225, 6.650), (1.225, 7.000), (1.400, 0.175), (1.400, 1.925), (1.400, 2.275), (1.400, 2.975), (1.400, 3.675), (1.400, 5.075), (1.400, 5.775), (1.400, 6.125), (1.400, 6.825), (1.575, 0.000), (1.575, 1.750), (1.575, 2.100), (1.575, 3.150), (1.575, 3.500), (1.575, 4.900), (1.575, 5.600), (1.575, 6.650), (1.575, 7.000), (1.750, 0.175), (1.750, 0.525), (1.750, 2.275), (1.750, 3.675), (1.750, 5.075), (1.750, 6.825), (1.925, 0.000), (1.925, 1.750), (1.925, 2.100), (1.925, 2.450), (1.925, 2.800), (1.925, 3.150), (1.925, 3.500), (1.925, 5.250), (1.925, 5.600), (1.925, 6.650), (1.925, 7.000), (2.100, 0.175), (2.100, 1.925), (2.100, 2.625), (2.100, 2.975), (2.100, 3.325), (2.100, 5.425), (2.275, 0.000), (2.275, 2.800), (2.275, 3.150), (2.275, 3.500), (2.275, 6.300), (2.275, 7.000), (2.450, 0.175), (2.450, 6.475), (2.450, 6.825), (2.625, 0.000), (2.625, 1.750), (2.625, 3.150), (2.625, 3.500), (2.625, 3.850), (2.625, 4.200), (2.625, 6.650), (2.625, 7.000), (2.800, 1.925), (2.800, 2.975), (2.800, 3.675), (2.800, 4.025), (2.800, 5.075), (2.800, 5.425), (2.975, 0.000), (2.975, 1.750), (2.975, 3.500), (2.975, 5.250), (2.975, 5.600), (2.975, 5.950), (2.975, 7.000), (3.150, 3.675), (3.150, 4.025), (3.150, 6.125), (3.325, 0.000), (3.325, 3.850), (3.325, 4.200), (3.325, 4.900), (3.325, 5.600), (3.325, 5.950), (3.325, 6.300), (3.325, 6.650), (3.325, 7.000), (3.500, 0.525), (3.500, 0.875), (3.500, 2.625), (3.500, 4.025), (3.500, 4.375), (3.500, 5.075), (3.500, 5.425), (3.500, 6.125), (3.500, 6.825), (3.675, 0.000), (3.675, 0.700), (3.675, 1.050), (3.675, 4.200), (3.675, 4.550), (3.675, 4.900), (3.675, 5.950), (3.675, 7.000), (3.850, 0.875), (3.850, 4.025), (3.850, 5.425), (3.850, 5.775), (4.025, 0.000), (4.025, 1.050), (4.025, 2.800), (4.025, 3.150), (4.025, 4.200), (4.025, 7.000), (4.200, 1.925), (4.200, 2.975), (4.200, 4.375), (4.200, 5.075), (4.200, 5.425), (4.200, 5.775), (4.200, 6.125), (4.375, 0.000), (4.375, 1.400), (4.375, 2.800), (4.375, 3.500), (4.375, 5.250), (4.375, 5.600), (4.375, 5.950), (4.375, 7.000), (4.550, 0.175), (4.550, 2.625), (4.550, 4.375), (4.550, 5.425), (4.550, 6.125), (4.725, 0.000), (4.725, 4.550), (4.725, 7.000), (4.900, 1.225), (4.900, 2.975), (4.900, 3.325), (5.075, 0.000), (5.075, 0.350), (5.075, 1.050), (5.075, 1.400), (5.075, 2.450), (5.075, 3.150), (5.075, 3.500), (5.075, 4.200), (5.075, 4.900), (5.075, 5.250), (5.075, 7.000), (5.250, 1.225), (5.250, 2.625), (5.250, 3.675), (5.250, 4.025), (5.250, 4.375), (5.250, 5.075), (5.425, 0.000), (5.425, 2.450), (5.425, 2.800), (5.425, 3.500), (5.425, 3.850), (5.425, 7.000), (5.600, 0.525), (5.600, 3.325), (5.600, 3.675), (5.775, 0.000), (5.775, 0.350), (5.775, 2.100), (5.775, 2.450), (5.775, 2.800), (5.775, 7.000), (5.950, 0.875), (5.950, 1.925), (5.950, 2.275), (5.950, 2.625), (5.950, 2.975), (5.950, 5.075), (5.950, 6.825), (6.125, 0.000), (6.125, 0.350), (6.125, 2.100), (6.125, 2.800), (6.125, 3.850), (6.125, 5.250), (6.125, 6.650), (6.125, 7.000), (6.300, 0.525), (6.300, 0.875), (6.300, 1.225), (6.300, 1.925), (6.300, 6.475), (6.300, 6.825), (6.475, 0.000), (6.475, 0.700), (6.475, 1.050), (6.475, 1.400), (6.475, 1.750), (6.475, 2.100), (6.475, 2.450), (6.475, 2.800), (6.475, 3.850), (6.475, 5.250), (6.475, 6.650), (6.475, 7.000), (6.650, 0.175), (6.650, 1.225), (6.650, 1.575), (6.650, 2.625), (6.650, 2.975), (6.650, 4.025), (6.650, 5.075), (6.650, 6.825), (6.825, 0.000), (6.825, 0.350), (6.825, 0.700), (6.825, 1.400), (6.825, 2.100), (6.825, 2.450), (6.825, 4.200), (6.825, 4.900), (6.825, 5.250), (6.825, 5.600), (6.825, 5.950), (6.825, 7.000), (7.000, 0.175), (7.000, 0.525), (7.000, 0.875), (7.000, 1.225), (7.000, 1.575), (7.000, 1.925), (7.000, 2.275), (7.000, 2.625), (7.000, 2.975), (7.000, 3.325), (7.000, 3.675), (7.000, 4.025), (7.000, 4.375), (7.000, 4.725), (7.000, 5.075), (7.000, 5.425), (7.000, 5.775), (7.000, 6.125), (7.000, 6.475), (7.000, 6.825)} {
\draw[fill = black] \p circle (0.087cm);
}

\foreach \p in {(0.350, 1.050), (0.350, 3.150), (0.350, 5.250), (0.350, 6.300), (0.350, 6.650), (0.525, 0.525), (0.525, 0.875), (0.525, 2.275), (0.525, 4.375), (0.525, 5.075), (0.525, 6.475), (0.700, 0.700), (0.700, 1.050), (0.700, 2.100), (0.700, 2.450), (0.700, 2.800), (0.700, 4.550), (0.700, 5.250), (0.875, 0.525), (0.875, 1.925), (0.875, 2.275), (0.875, 4.725), (0.875, 5.075), (0.875, 5.425), (1.050, 1.050), (1.050, 1.750), (1.050, 4.900), (1.225, 0.525), (1.225, 1.225), (1.225, 1.575), (1.400, 0.700), (1.400, 1.050), (1.400, 1.400), (1.400, 4.200), (1.575, 0.875), (1.575, 1.225), (1.575, 2.625), (1.575, 4.025), (1.750, 1.400), (1.750, 4.200), (1.750, 5.950), (1.925, 1.225), (1.925, 4.025), (1.925, 4.725), (1.925, 6.125), (2.100, 1.400), (2.100, 3.850), (2.100, 4.200), (2.275, 0.525), (2.275, 0.875), (2.275, 2.275), (2.275, 4.375), (2.275, 4.725), (2.275, 5.075), (2.275, 5.775), (2.450, 0.700), (2.450, 1.050), (2.450, 1.400), (2.450, 2.100), (2.450, 2.450), (2.450, 4.900), (2.450, 5.600), (2.450, 5.950), (2.625, 0.525), (2.625, 1.225), (2.625, 2.275), (2.625, 6.125), (2.800, 0.350), (2.800, 1.400), (2.800, 6.300), (2.975, 2.275), (2.975, 4.375), (2.975, 4.725), (2.975, 6.475), (3.150, 0.350), (3.150, 1.400), (3.150, 2.800), (3.150, 4.550), (3.325, 2.275), (3.325, 2.975), (3.500, 2.100), (3.500, 3.150), (3.500, 3.500), (3.675, 1.575), (3.675, 1.925), (3.675, 2.275), (3.675, 2.975), (3.675, 3.325), (3.675, 6.475), (3.850, 1.400), (3.850, 1.750), (3.850, 2.450), (3.850, 3.500), (3.850, 6.650), (4.025, 0.525), (4.025, 1.575), (4.025, 4.725), (4.200, 0.350), (4.200, 2.450), (4.200, 3.850), (4.375, 0.875), (4.375, 2.275), (4.375, 4.725), (4.375, 6.475), (4.550, 0.700), (4.550, 1.050), (4.550, 2.100), (4.550, 3.150), (4.550, 3.850), (4.550, 4.900), (4.550, 6.650), (4.725, 0.875), (4.725, 1.575), (4.725, 1.925), (4.725, 5.075), (4.725, 5.775), (4.900, 0.700), (4.900, 3.850), (4.900, 6.300), (5.075, 5.775), (5.075, 6.125), (5.250, 1.750), (5.250, 5.600), (5.250, 5.950), (5.250, 6.650), (5.425, 1.575), (5.425, 1.925), (5.425, 5.775), (5.425, 6.125), (5.600, 1.400), (5.600, 1.750), (5.600, 4.200), (5.600, 5.250), (5.600, 5.600), (5.600, 6.650), (5.775, 1.575), (5.775, 4.025), (5.775, 4.375), (5.775, 5.425), (5.775, 5.775), (5.950, 4.550), (6.125, 3.325), (6.125, 4.375), (6.125, 4.725), (6.125, 5.775), (6.300, 4.550), (6.300, 5.600), (6.300, 5.950), (6.475, 3.325), (6.475, 4.725), (6.475, 5.775), (6.650, 4.550)} {
\draw[fill = gray] \p circle (0.087cm);
}

\xdefinecolor{cutsetcolor}{RGB}{30,30,30}

\draw[line width = 0.017cm, color=cutsetcolor] (0.612,0.262) -- (0.612,0.262) -- (0.438,0.262) -- (0.438,0.438) -- (0.262,0.438) -- (0.262,0.612) -- (0.438,0.612) -- (0.438,0.787) -- (0.262,0.787) -- (0.262,0.962) -- (0.087,0.962) -- (0.087,1.138) -- (0.262,1.138) -- (0.262,1.313) -- (0.438,1.313) -- (0.438,1.487) -- (0.612,1.487) -- (0.612,1.662) -- (0.787,1.662) -- (0.787,1.837) -- (0.612,1.837) -- (0.612,2.012) -- (0.438,2.012) -- (0.438,2.188) -- (0.262,2.188) -- (0.262,2.362) -- (0.438,2.362) -- (0.438,2.537) -- (0.612,2.537) -- (0.612,2.712) -- (0.438,2.712) -- (0.438,2.887) -- (0.262,2.887) -- (0.262,3.063) -- (0.087,3.063) -- (0.087,3.237) -- (0.262,3.237) -- (0.262,3.412) -- (0.438,3.412) -- (0.438,3.237) -- (0.612,3.237) -- (0.612,3.063) -- (0.787,3.063) -- (0.787,2.887) -- (0.962,2.887) -- (0.962,2.712) -- (0.787,2.712) -- (0.787,2.537) -- (0.962,2.537) -- (0.962,2.362) -- (1.138,2.362) -- (1.138,2.188) -- (0.962,2.188) -- (0.962,2.012) -- (1.138,2.012) -- (1.138,1.837) -- (1.313,1.837) -- (1.313,1.662) -- (1.487,1.662) -- (1.487,1.487) -- (1.662,1.487) -- (1.662,1.662) -- (1.837,1.662) -- (1.837,1.487) -- (2.012,1.487) -- (2.012,1.662) -- (2.188,1.662) -- (2.188,1.837) -- (2.362,1.837) -- (2.362,2.012) -- (2.188,2.012) -- (2.188,2.188) -- (2.012,2.188) -- (2.012,2.362) -- (2.188,2.362) -- (2.188,2.537) -- (2.362,2.537) -- (2.362,2.712) -- (2.537,2.712) -- (2.537,2.887) -- (2.712,2.887) -- (2.712,2.712) -- (2.887,2.712) -- (2.887,2.887) -- (3.063,2.887) -- (3.063,3.063) -- (2.887,3.063) -- (2.887,3.237) -- (3.063,3.237) -- (3.063,3.412) -- (3.237,3.412) -- (3.237,3.587) -- (3.412,3.587) -- (3.412,3.762) -- (3.587,3.762) -- (3.587,3.937) -- (3.762,3.937) -- (3.762,3.762) -- (3.937,3.762) -- (3.937,3.937) -- (4.112,3.937) -- (4.112,4.112) -- (4.287,4.112) -- (4.287,4.287) -- (4.462,4.287) -- (4.462,4.112) -- (4.637,4.112) -- (4.637,4.287) -- (4.813,4.287) -- (4.813,4.462) -- (4.987,4.462) -- (4.987,4.637) -- (4.813,4.637) -- (4.813,4.813) -- (4.637,4.813) -- (4.637,4.637) -- (4.462,4.637) -- (4.462,4.462) -- (4.287,4.462) -- (4.287,4.637) -- (4.112,4.637) -- (4.112,4.462) -- (3.937,4.462) -- (3.937,4.637) -- (3.762,4.637) -- (3.762,4.813) -- (3.937,4.813) -- (3.937,4.987) -- (4.112,4.987) -- (4.112,4.813) -- (4.287,4.813) -- (4.287,4.987) -- (4.462,4.987) -- (4.462,5.162) -- (4.637,5.162) -- (4.637,5.337) -- (4.813,5.337) -- (4.813,5.512) -- (4.637,5.512) -- (4.637,5.688) -- (4.462,5.688) -- (4.462,5.862) -- (4.637,5.862) -- (4.637,6.037) -- (4.813,6.037) -- (4.813,6.212) -- (4.637,6.212) -- (4.637,6.387) -- (4.462,6.387) -- (4.462,6.212) -- (4.287,6.212) -- (4.287,6.387) -- (4.112,6.387) -- (4.112,6.212) -- (3.937,6.212) -- (3.937,6.037) -- (3.762,6.037) -- (3.762,6.212) -- (3.587,6.212) -- (3.587,6.387) -- (3.412,6.387) -- (3.412,6.563) -- (3.587,6.563) -- (3.587,6.737) -- (3.762,6.737) -- (3.762,6.912) -- (3.937,6.912) -- (3.937,6.737) -- (4.112,6.737) -- (4.112,6.912) -- (4.287,6.912) -- (4.287,6.737) -- (4.462,6.737) -- (4.462,6.912) -- (4.637,6.912) -- (4.637,6.737) -- (4.813,6.737) -- (4.813,6.912) -- (4.987,6.912) -- (4.987,6.737) -- (5.162,6.737) -- (5.162,6.912) -- (5.337,6.912) -- (5.337,6.737) -- (5.512,6.737) -- (5.512,6.912) -- (5.688,6.912) -- (5.688,6.737) -- (5.862,6.737) -- (5.862,6.563) -- (6.037,6.563) -- (6.037,6.387) -- (6.212,6.387) -- (6.212,6.212) -- (6.387,6.212) -- (6.387,6.387) -- (6.563,6.387) -- (6.563,6.563) -- (6.737,6.563) -- (6.737,6.387) -- (6.912,6.387) -- (6.912,6.212) -- (6.737,6.212) -- (6.737,6.037) -- (6.563,6.037) -- (6.563,5.862) -- (6.737,5.862) -- (6.737,5.688) -- (6.563,5.688) -- (6.563,5.512) -- (6.387,5.512) -- (6.387,5.337) -- (6.212,5.337) -- (6.212,5.512) -- (6.037,5.512) -- (6.037,5.337) -- (5.862,5.337) -- (5.862,5.162) -- (5.688,5.162) -- (5.688,4.987) -- (5.862,4.987) -- (5.862,4.813) -- (6.037,4.813) -- (6.037,4.987) -- (6.212,4.987) -- (6.212,5.162) -- (6.387,5.162) -- (6.387,4.987) -- (6.563,4.987) -- (6.563,4.813) -- (6.737,4.813) -- (6.737,4.637) -- (6.912,4.637) -- (6.912,4.462) -- (6.737,4.462) -- (6.737,4.287) -- (6.563,4.287) -- (6.563,4.112) -- (6.387,4.112) -- (6.387,3.937) -- (6.212,3.937) -- (6.212,4.112) -- (6.037,4.112) -- (6.037,3.937) -- (5.862,3.937) -- (5.862,3.762) -- (6.037,3.762) -- (6.037,3.587) -- (6.212,3.587) -- (6.212,3.762) -- (6.387,3.762) -- (6.387,3.587) -- (6.563,3.587) -- (6.563,3.762) -- (6.737,3.762) -- (6.737,3.587) -- (6.912,3.587) -- (6.912,3.412) -- (6.737,3.412) -- (6.737,3.237) -- (6.563,3.237) -- (6.563,3.063) -- (6.387,3.063) -- (6.387,2.887) -- (6.212,2.887) -- (6.212,3.063) -- (6.037,3.063) -- (6.037,3.237) -- (5.862,3.237) -- (5.862,3.063) -- (5.688,3.063) -- (5.688,2.887) -- (5.512,2.887) -- (5.512,3.063) -- (5.337,3.063) -- (5.337,2.887) -- (5.162,2.887) -- (5.162,2.712) -- (4.987,2.712) -- (4.987,2.537) -- (4.813,2.537) -- (4.813,2.362) -- (4.987,2.362) -- (4.987,2.188) -- (5.162,2.188) -- (5.162,2.362) -- (5.337,2.362) -- (5.337,2.188) -- (5.512,2.188) -- (5.512,2.012) -- (5.688,2.012) -- (5.688,1.837) -- (5.862,1.837) -- (5.862,1.662) -- (6.037,1.662) -- (6.037,1.837) -- (6.212,1.837) -- (6.212,1.662) -- (6.387,1.662) -- (6.387,1.487) -- (6.212,1.487) -- (6.212,1.313) -- (6.037,1.313) -- (6.037,1.138) -- (5.862,1.138) -- (5.862,0.962) -- (5.688,0.962) -- (5.688,0.787) -- (5.512,0.787) -- (5.512,0.612) -- (5.337,0.612) -- (5.337,0.438) -- (5.162,0.438) -- (5.162,0.612) -- (4.987,0.612) -- (4.987,0.438) -- (4.813,0.438) -- (4.813,0.262) -- (4.637,0.262) -- (4.637,0.438) -- (4.462,0.438) -- (4.462,0.262) -- (4.287,0.262) -- (4.287,0.087) -- (4.112,0.087) -- (4.112,0.262) -- (3.937,0.262) -- (3.937,0.087) -- (3.762,0.087) -- (3.762,0.262) -- (3.587,0.262) -- (3.587,0.087) -- (3.412,0.087) -- (3.412,0.262) -- (3.237,0.262) -- (3.237,0.087) -- (3.063,0.087) -- (3.063,0.262) -- (2.887,0.262) -- (2.887,0.087) -- (2.712,0.087) -- (2.712,0.262) -- (2.537,0.262) -- (2.537,0.438) -- (2.362,0.438) -- (2.362,0.262) -- (2.188,0.262) -- (2.188,0.438) -- (2.012,0.438) -- (2.012,0.612) -- (1.837,0.612) -- (1.837,0.787) -- (1.662,0.787) -- (1.662,0.612) -- (1.487,0.612) -- (1.487,0.438) -- (1.313,0.438) -- (1.313,0.262) -- (1.138,0.262) -- (1.138,0.438) -- (0.962,0.438) -- (0.962,0.262) -- (0.787,0.262) -- (0.787,0.087) -- (0.612,0.087) -- (0.612,0.262);

\draw[line width = 0.017cm, color=cutsetcolor] (0.612,3.762) -- (0.612,3.762) -- (0.438,3.762) -- (0.438,3.937) -- (0.262,3.937) -- (0.262,4.112) -- (0.438,4.112) -- (0.438,4.287) -- (0.262,4.287) -- (0.262,4.462) -- (0.438,4.462) -- (0.438,4.637) -- (0.612,4.637) -- (0.612,4.813) -- (0.438,4.813) -- (0.438,4.987) -- (0.262,4.987) -- (0.262,5.162) -- (0.087,5.162) -- (0.087,5.337) -- (0.262,5.337) -- (0.262,5.512) -- (0.438,5.512) -- (0.438,5.337) -- (0.612,5.337) -- (0.612,5.512) -- (0.787,5.512) -- (0.787,5.688) -- (0.962,5.688) -- (0.962,5.512) -- (1.138,5.512) -- (1.138,5.337) -- (1.313,5.337) -- (1.313,5.162) -- (1.138,5.162) -- (1.138,4.987) -- (1.313,4.987) -- (1.313,4.813) -- (1.487,4.813) -- (1.487,4.637) -- (1.662,4.637) -- (1.662,4.813) -- (1.837,4.813) -- (1.837,4.987) -- (2.012,4.987) -- (2.012,5.162) -- (2.188,5.162) -- (2.188,5.337) -- (2.362,5.337) -- (2.362,5.512) -- (2.188,5.512) -- (2.188,5.688) -- (2.012,5.688) -- (2.012,5.862) -- (1.837,5.862) -- (1.837,5.688) -- (1.662,5.688) -- (1.662,5.862) -- (1.487,5.862) -- (1.487,6.037) -- (1.662,6.037) -- (1.662,6.212) -- (1.487,6.212) -- (1.487,6.387) -- (1.662,6.387) -- (1.662,6.563) -- (1.837,6.563) -- (1.837,6.387) -- (2.012,6.387) -- (2.012,6.212) -- (2.188,6.212) -- (2.188,6.037) -- (2.362,6.037) -- (2.362,6.212) -- (2.537,6.212) -- (2.537,6.387) -- (2.712,6.387) -- (2.712,6.563) -- (2.887,6.563) -- (2.887,6.737) -- (3.063,6.737) -- (3.063,6.563) -- (3.237,6.563) -- (3.237,6.387) -- (3.063,6.387) -- (3.063,6.212) -- (2.887,6.212) -- (2.887,6.037) -- (2.712,6.037) -- (2.712,5.862) -- (2.887,5.862) -- (2.887,5.688) -- (2.712,5.688) -- (2.712,5.512) -- (2.537,5.512) -- (2.537,5.337) -- (2.712,5.337) -- (2.712,5.162) -- (2.537,5.162) -- (2.537,4.987) -- (2.712,4.987) -- (2.712,4.813) -- (2.887,4.813) -- (2.887,4.987) -- (3.063,4.987) -- (3.063,4.813) -- (3.237,4.813) -- (3.237,4.637) -- (3.412,4.637) -- (3.412,4.462) -- (3.237,4.462) -- (3.237,4.287) -- (3.063,4.287) -- (3.063,4.112) -- (2.887,4.112) -- (2.887,4.287) -- (2.712,4.287) -- (2.712,4.462) -- (2.537,4.462) -- (2.537,4.287) -- (2.362,4.287) -- (2.362,4.112) -- (2.537,4.112) -- (2.537,3.937) -- (2.362,3.937) -- (2.362,3.762) -- (2.188,3.762) -- (2.188,3.587) -- (2.012,3.587) -- (2.012,3.762) -- (1.837,3.762) -- (1.837,3.937) -- (1.662,3.937) -- (1.662,3.762) -- (1.487,3.762) -- (1.487,3.937) -- (1.313,3.937) -- (1.313,3.762) -- (1.138,3.762) -- (1.138,3.587) -- (0.962,3.587) -- (0.962,3.762) -- (0.787,3.762) -- (0.787,3.587) -- (0.612,3.587) -- (0.612,3.762);

\draw[line width = 0.017cm, color=cutsetcolor] (1.487,2.537) -- (1.487,2.537) -- (1.313,2.537) -- (1.313,2.712) -- (1.487,2.712) -- (1.487,2.887) -- (1.662,2.887) -- (1.662,2.712) -- (1.837,2.712) -- (1.837,2.537) -- (1.662,2.537) -- (1.662,2.362) -- (1.487,2.362) -- (1.487,2.537);

\draw[line width = 0.017cm, color=cutsetcolor] (0.262,6.212) -- (0.262,6.212) -- (0.087,6.212) -- (0.087,6.387) -- (0.262,6.387) -- (0.262,6.563) -- (0.087,6.563) -- (0.087,6.737) -- (0.262,6.737) -- (0.262,6.912) -- (0.438,6.912) -- (0.438,6.737) -- (0.612,6.737) -- (0.612,6.563) -- (0.787,6.563) -- (0.787,6.387) -- (0.612,6.387) -- (0.612,6.212) -- (0.438,6.212) -- (0.438,6.037) -- (0.262,6.037) -- (0.262,6.212);

\end{tikzpicture}

%% file: fig-simulation-41x41-5.0.tex
\begin{tikzpicture}
\foreach \p in {(0.000, 0.175), (0.000, 0.525), (0.000, 0.875), (0.000, 1.225), (0.000, 1.575), (0.000, 1.925), (0.000, 2.275), (0.000, 2.625), (0.000, 2.975), (0.000, 3.325), (0.000, 3.675), (0.000, 4.025), (0.000, 4.375), (0.000, 4.725), (0.000, 5.075), (0.000, 5.425), (0.000, 5.775), (0.000, 6.125), (0.000, 6.475), (0.000, 6.825), (0.175, 0.000), (0.175, 0.350), (0.175, 0.700), (0.175, 1.050), (0.175, 1.400), (0.175, 1.750), (0.175, 2.450), (0.175, 2.800), (0.175, 3.150), (0.175, 3.500), (0.175, 3.850), (0.175, 4.200), (0.175, 4.550), (0.175, 5.250), (0.175, 5.600), (0.175, 6.300), (0.175, 6.650), (0.175, 7.000), (0.350, 0.175), (0.350, 0.525), (0.350, 0.875), (0.350, 1.225), (0.350, 1.575), (0.350, 1.925), (0.350, 2.275), (0.350, 2.625), (0.350, 2.975), (0.350, 3.325), (0.350, 4.025), (0.350, 4.725), (0.350, 5.075), (0.350, 5.425), (0.350, 5.775), (0.350, 6.125), (0.350, 6.475), (0.350, 6.825), (0.525, 0.000), (0.525, 0.700), (0.525, 1.400), (0.525, 2.100), (0.525, 2.450), (0.525, 3.150), (0.525, 3.500), (0.525, 4.200), (0.525, 4.900), (0.525, 5.250), (0.525, 5.600), (0.525, 5.950), (0.525, 6.650), (0.525, 7.000), (0.700, 0.175), (0.700, 0.875), (0.700, 1.225), (0.700, 1.575), (0.700, 1.925), (0.700, 2.275), (0.700, 2.625), (0.700, 2.975), (0.700, 3.325), (0.700, 3.675), (0.700, 4.375), (0.700, 4.725), (0.700, 5.425), (0.700, 5.775), (0.700, 6.125), (0.700, 6.475), (0.700, 6.825), (0.875, 0.000), (0.875, 0.350), (0.875, 0.700), (0.875, 1.050), (0.875, 1.750), (0.875, 2.100), (0.875, 2.450), (0.875, 2.800), (0.875, 3.150), (0.875, 3.500), (0.875, 4.550), (0.875, 4.900), (0.875, 5.250), (0.875, 5.600), (0.875, 5.950), (0.875, 6.650), (0.875, 7.000), (1.050, 0.175), (1.050, 0.525), (1.050, 0.875), (1.050, 1.225), (1.050, 1.575), (1.050, 1.925), (1.050, 2.275), (1.050, 2.975), (1.050, 3.325), (1.050, 3.675), (1.050, 4.025), (1.050, 4.375), (1.050, 5.075), (1.050, 5.775), (1.050, 6.125), (1.050, 6.475), (1.050, 6.825), (1.225, 0.000), (1.225, 0.350), (1.225, 0.700), (1.225, 1.050), (1.225, 1.400), (1.225, 2.450), (1.225, 3.150), (1.225, 3.500), (1.225, 3.850), (1.225, 4.200), (1.225, 4.550), (1.225, 4.900), (1.225, 5.250), (1.225, 5.600), (1.225, 5.950), (1.225, 6.650), (1.225, 7.000), (1.400, 0.175), (1.400, 0.525), (1.400, 0.875), (1.400, 1.225), (1.400, 1.575), (1.400, 1.925), (1.400, 2.275), (1.400, 2.625), (1.400, 2.975), (1.400, 3.325), (1.400, 3.675), (1.400, 4.025), (1.400, 5.425), (1.400, 5.775), (1.400, 6.125), (1.400, 6.475), (1.400, 6.825), (1.575, 0.000), (1.575, 0.350), (1.575, 1.050), (1.575, 1.400), (1.575, 1.750), (1.575, 2.100), (1.575, 2.450), (1.575, 2.800), (1.575, 3.150), (1.575, 3.500), (1.575, 3.850), (1.575, 4.200), (1.575, 4.550), (1.575, 4.900), (1.575, 5.600), (1.575, 5.950), (1.575, 7.000), (1.750, 0.175), (1.750, 0.525), (1.750, 0.875), (1.750, 1.225), (1.750, 2.275), (1.750, 2.625), (1.750, 2.975), (1.750, 3.325), (1.750, 3.675), (1.750, 4.025), (1.750, 4.375), (1.750, 4.725), (1.750, 5.075), (1.750, 5.425), (1.750, 5.775), (1.750, 6.125), (1.750, 6.475), (1.750, 6.825), (1.925, 0.000), (1.925, 0.350), (1.925, 0.700), (1.925, 1.050), (1.925, 1.750), (1.925, 2.100), (1.925, 2.450), (1.925, 2.800), (1.925, 3.150), (1.925, 3.500), (1.925, 3.850), (1.925, 4.550), (1.925, 4.900), (1.925, 5.250), (1.925, 5.600), (1.925, 5.950), (1.925, 6.300), (1.925, 6.650), (1.925, 7.000), (2.100, 0.175), (2.100, 0.875), (2.100, 1.225), (2.100, 1.575), (2.100, 1.925), (2.100, 2.275), (2.100, 2.975), (2.100, 3.325), (2.100, 4.025), (2.100, 4.375), (2.100, 4.725), (2.100, 5.075), (2.100, 5.775), (2.100, 6.125), (2.100, 6.475), (2.275, 0.000), (2.275, 0.700), (2.275, 1.050), (2.275, 1.400), (2.275, 1.750), (2.275, 2.100), (2.275, 2.450), (2.275, 2.800), (2.275, 3.150), (2.275, 3.500), (2.275, 3.850), (2.275, 4.200), (2.275, 4.550), (2.275, 4.900), (2.275, 5.950), (2.275, 6.300), (2.275, 6.650), (2.275, 7.000), (2.450, 0.525), (2.450, 0.875), (2.450, 1.225), (2.450, 1.925), (2.450, 2.625), (2.450, 2.975), (2.450, 3.325), (2.450, 3.675), (2.450, 4.025), (2.450, 4.375), (2.450, 4.725), (2.450, 5.075), (2.450, 5.775), (2.450, 6.125), (2.450, 6.475), (2.450, 6.825), (2.625, 0.000), (2.625, 0.350), (2.625, 0.700), (2.625, 1.050), (2.625, 1.400), (2.625, 1.750), (2.625, 2.100), (2.625, 2.450), (2.625, 2.800), (2.625, 3.500), (2.625, 4.200), (2.625, 4.550), (2.625, 4.900), (2.625, 5.600), (2.625, 5.950), (2.625, 6.300), (2.625, 6.650), (2.625, 7.000), (2.800, 0.525), (2.800, 0.875), (2.800, 1.225), (2.800, 1.575), (2.800, 2.625), (2.800, 2.975), (2.800, 3.325), (2.800, 3.675), (2.800, 4.025), (2.800, 4.375), (2.800, 4.725), (2.800, 5.075), (2.800, 6.125), (2.800, 6.475), (2.975, 0.000), (2.975, 0.350), (2.975, 0.700), (2.975, 1.050), (2.975, 1.400), (2.975, 2.100), (2.975, 2.450), (2.975, 2.800), (2.975, 3.150), (2.975, 3.500), (2.975, 3.850), (2.975, 4.200), (2.975, 4.550), (2.975, 4.900), (2.975, 5.250), (2.975, 5.600), (2.975, 5.950), (2.975, 6.300), (2.975, 6.650), (2.975, 7.000), (3.150, 0.175), (3.150, 0.525), (3.150, 0.875), (3.150, 1.225), (3.150, 2.275), (3.150, 2.975), (3.150, 3.325), (3.150, 3.675), (3.150, 4.025), (3.150, 4.375), (3.150, 4.725), (3.150, 5.075), (3.150, 5.425), (3.150, 5.775), (3.150, 6.125), (3.150, 6.475), (3.150, 6.825), (3.325, 0.000), (3.325, 0.350), (3.325, 0.700), (3.325, 2.100), (3.325, 2.450), (3.325, 2.800), (3.325, 3.150), (3.325, 3.500), (3.325, 3.850), (3.325, 4.200), (3.325, 4.550), (3.325, 4.900), (3.325, 5.950), (3.325, 6.300), (3.325, 6.650), (3.325, 7.000), (3.500, 0.175), (3.500, 0.875), (3.500, 1.225), (3.500, 2.275), (3.500, 2.625), (3.500, 2.975), (3.500, 3.675), (3.500, 4.025), (3.500, 4.375), (3.500, 5.075), (3.500, 5.775), (3.500, 6.125), (3.500, 6.825), (3.675, 0.000), (3.675, 0.700), (3.675, 1.050), (3.675, 2.100), (3.675, 2.450), (3.675, 2.800), (3.675, 3.150), (3.675, 3.500), (3.675, 3.850), (3.675, 4.200), (3.675, 4.550), (3.675, 4.900), (3.675, 5.250), (3.675, 5.600), (3.675, 5.950), (3.675, 6.300), (3.675, 6.650), (3.675, 7.000), (3.850, 0.175), (3.850, 0.525), (3.850, 0.875), (3.850, 1.225), (3.850, 1.925), (3.850, 2.275), (3.850, 2.625), (3.850, 2.975), (3.850, 3.325), (3.850, 3.675), (3.850, 4.025), (3.850, 4.375), (3.850, 4.725), (3.850, 5.075), (3.850, 5.425), (3.850, 5.775), (3.850, 6.125), (3.850, 6.825), (4.025, 0.000), (4.025, 0.350), (4.025, 0.700), (4.025, 1.050), (4.025, 1.400), (4.025, 1.750), (4.025, 2.100), (4.025, 2.450), (4.025, 2.800), (4.025, 3.500), (4.025, 3.850), (4.025, 4.200), (4.025, 4.550), (4.025, 4.900), (4.025, 5.250), (4.025, 5.600), (4.025, 5.950), (4.025, 6.300), (4.025, 6.650), (4.025, 7.000), (4.200, 0.175), (4.200, 0.525), (4.200, 0.875), (4.200, 1.225), (4.200, 1.575), (4.200, 2.275), (4.200, 2.625), (4.200, 2.975), (4.200, 3.325), (4.200, 3.675), (4.200, 4.025), (4.200, 4.375), (4.200, 4.725), (4.200, 5.075), (4.200, 5.425), (4.200, 5.775), (4.200, 6.475), (4.200, 6.825), (4.375, 0.000), (4.375, 0.350), (4.375, 0.700), (4.375, 1.050), (4.375, 1.400), (4.375, 2.100), (4.375, 2.450), (4.375, 2.800), (4.375, 3.150), (4.375, 3.500), (4.375, 3.850), (4.375, 4.200), (4.375, 4.550), (4.375, 4.900), (4.375, 5.250), (4.375, 5.600), (4.375, 5.950), (4.375, 6.300), (4.375, 6.650), (4.375, 7.000), (4.550, 0.175), (4.550, 0.525), (4.550, 0.875), (4.550, 1.575), (4.550, 1.925), (4.550, 2.625), (4.550, 2.975), (4.550, 3.325), (4.550, 4.025), (4.550, 4.375), (4.550, 4.725), (4.550, 5.075), (4.550, 5.425), (4.550, 5.775), (4.550, 6.125), (4.550, 6.825), (4.725, 0.000), (4.725, 0.350), (4.725, 0.700), (4.725, 1.050), (4.725, 1.400), (4.725, 1.750), (4.725, 3.150), (4.725, 3.850), (4.725, 4.200), (4.725, 4.550), (4.725, 4.900), (4.725, 5.250), (4.725, 5.600), (4.725, 5.950), (4.725, 6.300), (4.725, 6.650), (4.725, 7.000), (4.900, 0.525), (4.900, 1.225), (4.900, 1.575), (4.900, 3.325), (4.900, 3.675), (4.900, 4.025), (4.900, 4.375), (4.900, 4.725), (4.900, 5.075), (4.900, 5.425), (4.900, 5.775), (4.900, 6.475), (4.900, 6.825), (5.075, 0.000), (5.075, 0.350), (5.075, 0.700), (5.075, 1.050), (5.075, 1.400), (5.075, 1.750), (5.075, 3.500), (5.075, 3.850), (5.075, 4.200), (5.075, 4.550), (5.075, 4.900), (5.075, 5.600), (5.075, 5.950), (5.075, 6.300), (5.075, 6.650), (5.075, 7.000), (5.250, 0.175), (5.250, 0.525), (5.250, 1.225), (5.250, 3.325), (5.250, 4.375), (5.250, 4.725), (5.250, 5.425), (5.250, 5.775), (5.250, 6.125), (5.250, 6.475), (5.250, 6.825), (5.425, 0.000), (5.425, 0.350), (5.425, 0.700), (5.425, 1.050), (5.425, 1.400), (5.425, 1.750), (5.425, 2.450), (5.425, 3.150), (5.425, 3.500), (5.425, 3.850), (5.425, 4.200), (5.425, 4.550), (5.425, 5.600), (5.425, 5.950), (5.425, 6.300), (5.425, 6.650), (5.425, 7.000), (5.600, 0.175), (5.600, 0.525), (5.600, 0.875), (5.600, 1.225), (5.600, 1.575), (5.600, 1.925), (5.600, 2.275), (5.600, 2.975), (5.600, 3.325), (5.600, 3.675), (5.600, 4.025), (5.600, 4.375), (5.600, 4.725), (5.600, 5.425), (5.600, 5.775), (5.600, 6.125), (5.600, 6.475), (5.600, 6.825), (5.775, 0.000), (5.775, 0.350), (5.775, 0.700), (5.775, 1.750), (5.775, 2.100), (5.775, 2.450), (5.775, 2.800), (5.775, 3.150), (5.775, 3.500), (5.775, 3.850), (5.775, 4.550), (5.775, 4.900), (5.775, 5.250), (5.775, 5.600), (5.775, 6.650), (5.775, 7.000), (5.950, 0.175), (5.950, 0.525), (5.950, 0.875), (5.950, 1.575), (5.950, 1.925), (5.950, 2.275), (5.950, 2.625), (5.950, 2.975), (5.950, 3.325), (5.950, 3.675), (5.950, 4.025), (5.950, 4.375), (5.950, 4.725), (5.950, 5.075), (5.950, 5.425), (5.950, 5.775), (5.950, 6.125), (5.950, 6.475), (5.950, 6.825), (6.125, 0.000), (6.125, 0.350), (6.125, 0.700), (6.125, 1.050), (6.125, 1.400), (6.125, 1.750), (6.125, 2.100), (6.125, 2.450), (6.125, 2.800), (6.125, 3.150), (6.125, 3.500), (6.125, 3.850), (6.125, 4.200), (6.125, 4.550), (6.125, 4.900), (6.125, 5.250), (6.125, 5.600), (6.125, 6.300), (6.125, 6.650), (6.125, 7.000), (6.300, 0.175), (6.300, 0.875), (6.300, 1.225), (6.300, 1.575), (6.300, 1.925), (6.300, 2.275), (6.300, 2.625), (6.300, 2.975), (6.300, 3.325), (6.300, 4.375), (6.300, 4.725), (6.300, 5.075), (6.300, 5.425), (6.300, 5.775), (6.300, 6.125), (6.300, 6.825), (6.475, 0.000), (6.475, 0.350), (6.475, 0.700), (6.475, 1.050), (6.475, 1.400), (6.475, 2.100), (6.475, 2.800), (6.475, 3.150), (6.475, 3.500), (6.475, 3.850), (6.475, 4.200), (6.475, 4.550), (6.475, 4.900), (6.475, 5.250), (6.475, 5.600), (6.475, 5.950), (6.475, 7.000), (6.650, 0.175), (6.650, 0.525), (6.650, 0.875), (6.650, 1.225), (6.650, 1.575), (6.650, 1.925), (6.650, 2.275), (6.650, 2.625), (6.650, 2.975), (6.650, 3.325), (6.650, 4.025), (6.650, 4.375), (6.650, 4.725), (6.650, 5.075), (6.650, 5.775), (6.650, 6.125), (6.650, 6.825), (6.825, 0.000), (6.825, 0.350), (6.825, 0.700), (6.825, 1.050), (6.825, 1.750), (6.825, 2.100), (6.825, 3.150), (6.825, 3.500), (6.825, 3.850), (6.825, 4.550), (6.825, 4.900), (6.825, 5.600), (6.825, 5.950), (6.825, 7.000), (7.000, 0.175), (7.000, 0.525), (7.000, 0.875), (7.000, 1.225), (7.000, 1.575), (7.000, 1.925), (7.000, 2.275), (7.000, 2.625), (7.000, 2.975), (7.000, 3.325), (7.000, 3.675), (7.000, 4.025), (7.000, 4.375), (7.000, 4.725), (7.000, 5.075), (7.000, 5.425), (7.000, 5.775), (7.000, 6.125), (7.000, 6.475), (7.000, 6.825)} {
\draw[fill = black] \p circle (0.087cm);
}

\foreach \p in {(2.275, 5.425), (3.150, 1.750), (3.325, 1.575), (3.500, 1.750), (3.675, 1.575), (4.725, 2.275), (4.900, 2.100), (4.900, 2.450), (4.900, 2.800), (5.075, 2.275), (5.075, 2.625), (5.075, 2.975), (5.250, 2.100), (5.250, 2.800), (5.425, 5.075), (6.475, 6.475)} {
\draw[fill = gray] \p circle (0.087cm);
}

\xdefinecolor{cutsetcolor}{RGB}{30,30,30}

\draw[line width = 0.017cm, color=cutsetcolor] (6.387,6.387) -- (6.387,6.387) -- (6.212,6.387) -- (6.212,6.563) -- (6.387,6.563) -- (6.387,6.737) -- (6.563,6.737) -- (6.563,6.563) -- (6.737,6.563) -- (6.737,6.387) -- (6.563,6.387) -- (6.563,6.212) -- (6.387,6.212) -- (6.387,6.387);
\xdefinecolor{cutsetcolor}{RGB}{30,30,30}

\draw[line width = 0.017cm, color=cutsetcolor] (5.337,4.987) -- (5.337,4.987) -- (5.162,4.987) -- (5.162,5.162) -- (5.337,5.162) -- (5.337,5.337) -- (5.512,5.337) -- (5.512,5.162) -- (5.688,5.162) -- (5.688,4.987) -- (5.512,4.987) -- (5.512,4.813) -- (5.337,4.813) -- (5.337,4.987);
\xdefinecolor{cutsetcolor}{RGB}{30,30,30}

\draw[line width = 0.017cm, color=cutsetcolor] (4.813,2.012) -- (4.813,2.012) -- (4.637,2.012) -- (4.637,2.188) -- (4.462,2.188) -- (4.462,2.362) -- (4.637,2.362) -- (4.637,2.537) -- (4.813,2.537) -- (4.813,2.712) -- (4.637,2.712) -- (4.637,2.887) -- (4.813,2.887) -- (4.813,3.063) -- (4.987,3.063) -- (4.987,3.237) -- (5.162,3.237) -- (5.162,3.063) -- (5.337,3.063) -- (5.337,2.887) -- (5.512,2.887) -- (5.512,2.712) -- (5.337,2.712) -- (5.337,2.537) -- (5.162,2.537) -- (5.162,2.362) -- (5.337,2.362) -- (5.337,2.188) -- (5.512,2.188) -- (5.512,2.012) -- (5.337,2.012) -- (5.337,1.837) -- (5.162,1.837) -- (5.162,2.012) -- (4.987,2.012) -- (4.987,1.837) -- (4.813,1.837) -- (4.813,2.012);
\xdefinecolor{cutsetcolor}{RGB}{30,30,30}

\draw[line width = 0.017cm, color=cutsetcolor] (3.237,1.487) -- (3.237,1.487) -- (3.063,1.487) -- (3.063,1.662) -- (2.887,1.662) -- (2.887,1.837) -- (3.063,1.837) -- (3.063,2.012) -- (3.237,2.012) -- (3.237,1.837) -- (3.412,1.837) -- (3.412,2.012) -- (3.587,2.012) -- (3.587,1.837) -- (3.762,1.837) -- (3.762,1.662) -- (3.937,1.662) -- (3.937,1.487) -- (3.762,1.487) -- (3.762,1.313) -- (3.587,1.313) -- (3.587,1.487) -- (3.412,1.487) -- (3.412,1.313) -- (3.237,1.313) -- (3.237,1.487);
\xdefinecolor{cutsetcolor}{RGB}{30,30,30}

\draw[line width = 0.017cm, color=cutsetcolor] (2.188,5.337) -- (2.188,5.337) -- (2.012,5.337) -- (2.012,5.512) -- (2.188,5.512) -- (2.188,5.688) -- (2.362,5.688) -- (2.362,5.512) -- (2.537,5.512) -- (2.537,5.337) -- (2.362,5.337) -- (2.362,5.162) -- (2.188,5.162) -- (2.188,5.337);
\end{tikzpicture}

%% file: fig-pre-shift.tex
\begin{tikzpicture}
  \xdefinecolor{cutsetcolor}{RGB}{30,30,30}

  \draw[line width = 0.015cm, color=cutsetcolor] (8.531,18.156) -- (8.531,18.156) -- (8.094,18.156) -- (8.094,18.594) -- (8.531,18.594) -- (8.531,19.031) -- (8.094,19.031) -- (8.094,19.469) -- (7.656,19.469) -- (7.656,19.906) -- (8.094,19.906) -- (8.094,20.344) -- (8.531,20.344) -- (8.531,20.781) -- (8.969,20.781) -- (8.969,21.219) -- (9.406,21.219) -- (9.406,20.781) -- (9.844,20.781) -- (9.844,21.219) -- (10.281,21.219) -- (10.281,20.781) -- (10.719,20.781) -- (10.719,21.219) -- (11.156,21.219) -- (11.156,21.656) -- (11.594,21.656) -- (11.594,21.219) -- (12.031,21.219) -- (12.031,20.781) -- (12.469,20.781) -- (12.469,20.344) -- (12.906,20.344) -- (12.906,19.906) -- (12.469,19.906) -- (12.469,19.469) -- (12.031,19.469) -- (12.031,19.906) -- (11.594,19.906) -- (11.594,20.344) -- (11.156,20.344) -- (11.156,19.906) -- (10.719,19.906) -- (10.719,19.469) -- (10.281,19.469) -- (10.281,19.031) -- (9.844,19.031) -- (9.844,18.594) -- (9.406,18.594) -- (9.406,18.156) -- (8.969,18.156) -- (8.969,17.719) -- (8.531,17.719) -- (8.531,18.156);
  \foreach \p in {(8.313, 19.688), (8.750, 18.375), (8.750, 19.250), (8.750, 20.125), (9.188, 18.813), (9.188, 19.688), (9.188, 20.563), (9.625, 19.250), (10.063, 19.688), (10.063, 20.563), (10.938, 20.563), (11.375, 21.000), (11.813, 20.563), (12.250, 20.125)} {
    \draw[fill = gray] \p circle (0.175cm);
  }

  \foreach \p in {(7.438, 17.500), (7.438, 18.375), (7.438, 19.250), (7.438, 20.125), (7.438, 21.000), (7.438, 21.875), (7.875, 18.813), (7.875, 20.563), (7.875, 21.438), (8.313, 17.500), (8.313, 21.000), (8.313, 21.875), (8.750, 21.438), (9.188, 17.500), (9.188, 21.875), (9.625, 17.938), (9.625, 21.438), (10.063, 17.500), (10.063, 18.375), (10.063, 21.875), (10.500, 17.938), (10.500, 18.813), (10.500, 21.438), (10.938, 18.375), (10.938, 19.250), (10.938, 21.875), (11.375, 17.938), (11.375, 18.813), (11.375, 19.688), (11.813, 19.250), (11.813, 21.875), (12.250, 21.438), (12.688, 17.500), (12.688, 19.250), (12.688, 21.000), (12.688, 21.875), (13.125, 17.938), (13.125, 19.688), (13.125, 20.563), (13.125, 21.438)} {
    \draw[fill = black] \p circle (0.175cm);
  }

\end{tikzpicture}

%% file: fig-post-shift.tex
\begin{tikzpicture}

  \xdefinecolor{cutsetcolor}{RGB}{30,30,30}

  \draw[line width = 0.015cm, color=cutsetcolor] (8.531,18.156) -- (8.531,18.156) -- (8.094,18.156) -- (8.094,18.594) -- (8.531,18.594) -- (8.531,19.031) -- (8.094,19.031) -- (8.094,19.469) -- (7.656,19.469) -- (7.656,19.906) -- (8.094,19.906) -- (8.094,20.344) -- (8.531,20.344) -- (8.531,20.781) -- (8.969,20.781) -- (8.969,21.219) -- (9.406,21.219) -- (9.406,20.781) -- (9.844,20.781) -- (9.844,21.219) -- (10.281,21.219) -- (10.281,20.781) -- (10.719,20.781) -- (10.719,21.219) -- (11.156,21.219) -- (11.156,21.656) -- (11.594,21.656) -- (11.594,21.219) -- (12.031,21.219) -- (12.031,20.781) -- (12.469,20.781) -- (12.469,20.344) -- (12.906,20.344) -- (12.906,19.906) -- (12.469,19.906) -- (12.469,19.469) -- (12.031,19.469) -- (12.031,19.906) -- (11.594,19.906) -- (11.594,20.344) -- (11.156,20.344) -- (11.156,19.906) -- (10.719,19.906) -- (10.719,19.469) -- (10.281,19.469) -- (10.281,19.031) -- (9.844,19.031) -- (9.844,18.594) -- (9.406,18.594) -- (9.406,18.156) -- (8.969,18.156) -- (8.969,17.719) -- (8.531,17.719) -- (8.531,18.156);
  \foreach \p in {} {
    \draw[fill = gray] \p circle (0.175cm);
  }
  
  \foreach \p in {(7.438, 17.500), (7.438, 18.375), (7.438, 19.250), (7.438, 20.125), (7.438, 21.000), (7.438, 21.875), (7.875, 18.813), (7.875, 19.688), (7.875, 20.563), (7.875, 21.438), (8.313, 17.500), (8.313, 18.375), (8.313, 19.250), (8.313, 20.125), (8.313, 21.000), (8.313, 21.875), (8.750, 18.813), (8.750, 19.688), (8.750, 20.563), (8.750, 21.438), (9.188, 17.500), (9.188, 19.250), (9.188, 21.875), (9.625, 17.938), (9.625, 19.688), (9.625, 20.563), (9.625, 21.438), (10.063, 17.500), (10.063, 18.375), (10.063, 21.875), (10.500, 17.938), (10.500, 18.813), (10.500, 20.563), (10.500, 21.438), (10.938, 18.375), (10.938, 19.250), (10.938, 21.000), (10.938, 21.875), (11.375, 17.938), (11.375, 18.813), (11.375, 19.688), (11.375, 20.563), (11.813, 19.250), (11.813, 20.125), (11.813, 21.875), (12.250, 21.438), (12.688, 17.500), (12.688, 19.250), (12.688, 21.000), (12.688, 21.875), (13.125, 17.938), (13.125, 19.688), (13.125, 20.563), (13.125, 21.438)} {
    \draw[fill = black] \p circle (0.175cm);
  }
  
  \foreach \p in {(8.750, 17.938), (9.188, 18.375), (9.188, 21.000), (9.625, 18.813), (10.063, 19.250), (10.063, 21.000), (10.500, 19.688), (10.938, 20.125), (11.375, 21.438), (11.813, 21.000), (12.250, 19.688), (12.250, 20.563), (12.688, 20.125)} {
    \draw[fill = white] \p circle (0.175cm);
  }

\end{tikzpicture}

%% file: fig-cube-0.tex
\begin{tikzpicture}

  \path[use as bounding box] (-0.250,-0.250) rectangle (7.000,7.000);

  \draw[line width = 0.2, color = black, densely dotted] (1.000,0.500) -- (1.000,0.250) -- (1.250,0.250) -- (1.250,0.500);
  \draw[line width = 0.2, color = black, densely dotted] (1.500,0.500) -- (1.500,0.250) -- (1.750,0.250) -- (1.750,0.500);
  \draw[line width = 0.2, color = black, densely dotted] (2.000,0.500) -- (2.000,0.250) -- (2.250,0.250) -- (2.250,0.500);
  \draw[line width = 0.2, color = black, densely dotted] (2.500,0.500) -- (2.500,0.250) -- (2.750,0.250) -- (2.750,0.500);
  \draw[line width = 0.2, color = black, densely dotted] (3.000,0.500) -- (3.000,0.250) -- (3.250,0.250) -- (3.250,0.500);
  \draw[line width = 0.2, color = black, densely dotted] (3.500,0.500) -- (3.500,0.250) -- (3.750,0.250) -- (3.750,0.500);
  \draw[line width = 0.2, color = black, densely dotted] (4.000,0.500) -- (4.000,0.250) -- (4.250,0.250) -- (4.250,0.500);
  \draw[line width = 0.2, color = black, densely dotted] (4.500,0.500) -- (4.500,0.250) -- (4.750,0.250) -- (4.750,0.500);
  \draw[line width = 0.2, color = black, densely dotted] (5.000,0.500) -- (5.000,0.250) -- (5.250,0.250) -- (5.250,0.500);
  \draw[line width = 0.2, color = black, densely dotted] (5.500,0.500) -- (5.500,0.250) -- (5.750,0.250) -- (5.750,0.500);
  \draw[line width = 0.2, color = black, densely dotted] (1.000,6.250) -- (1.000,6.500) -- (1.250,6.500) -- (1.250,6.250);
  \draw[line width = 0.2, color = black, densely dotted] (1.500,6.250) -- (1.500,6.500) -- (1.750,6.500) -- (1.750,6.250);
  \draw[line width = 0.2, color = black, densely dotted] (2.000,6.250) -- (2.000,6.500) -- (2.250,6.500) -- (2.250,6.250);
  \draw[line width = 0.2, color = black, densely dotted] (2.500,6.250) -- (2.500,6.500) -- (2.750,6.500) -- (2.750,6.250);
  \draw[line width = 0.2, color = black, densely dotted] (3.000,6.250) -- (3.000,6.500) -- (3.250,6.500) -- (3.250,6.250);
  \draw[line width = 0.2, color = black, densely dotted] (3.500,6.250) -- (3.500,6.500) -- (3.750,6.500) -- (3.750,6.250);
  \draw[line width = 0.2, color = black, densely dotted] (4.000,6.250) -- (4.000,6.500) -- (4.250,6.500) -- (4.250,6.250);
  \draw[line width = 0.2, color = black, densely dotted] (4.500,6.250) -- (4.500,6.500) -- (4.750,6.500) -- (4.750,6.250);
  \draw[line width = 0.2, color = black, densely dotted] (5.000,6.250) -- (5.000,6.500) -- (5.250,6.500) -- (5.250,6.250);
  \draw[line width = 0.2, color = black, densely dotted] (5.500,6.250) -- (5.500,6.500) -- (5.750,6.500) -- (5.750,6.250);
  \draw[line width = 0.2, color = black, densely dotted] (0.500,1.000) -- (0.250,1.000) -- (0.250,1.250) -- (0.500,1.250);
  \draw[line width = 0.2, color = black, densely dotted] (0.500,1.500) -- (0.250,1.500) -- (0.250,1.750) -- (0.500,1.750);
  \draw[line width = 0.2, color = black, densely dotted] (0.500,2.000) -- (0.250,2.000) -- (0.250,2.250) -- (0.500,2.250);
  \draw[line width = 0.2, color = black, densely dotted] (0.500,2.500) -- (0.250,2.500) -- (0.250,2.750) -- (0.500,2.750);
  \draw[line width = 0.2, color = black, densely dotted] (0.500,3.000) -- (0.250,3.000) -- (0.250,3.250) -- (0.500,3.250);
  \draw[line width = 0.2, color = black, densely dotted] (0.500,3.500) -- (0.250,3.500) -- (0.250,3.750) -- (0.500,3.750);
  \draw[line width = 0.2, color = black, densely dotted] (0.500,4.000) -- (0.250,4.000) -- (0.250,4.250) -- (0.500,4.250);
  \draw[line width = 0.2, color = black, densely dotted] (0.500,4.500) -- (0.250,4.500) -- (0.250,4.750) -- (0.500,4.750);
  \draw[line width = 0.2, color = black, densely dotted] (0.500,5.000) -- (0.250,5.000) -- (0.250,5.250) -- (0.500,5.250);
  \draw[line width = 0.2, color = black, densely dotted] (0.500,5.500) -- (0.250,5.500) -- (0.250,5.750) -- (0.500,5.750);
  \draw[line width = 0.2, color = black, densely dotted] (6.250,1.000) -- (6.500,1.000) -- (6.500,1.250) -- (6.250,1.250);
  \draw[line width = 0.2, color = black, densely dotted] (6.250,1.500) -- (6.500,1.500) -- (6.500,1.750) -- (6.250,1.750);
  \draw[line width = 0.2, color = black, densely dotted] (6.250,2.000) -- (6.500,2.000) -- (6.500,2.250) -- (6.250,2.250);
  \draw[line width = 0.2, color = black, densely dotted] (6.250,2.500) -- (6.500,2.500) -- (6.500,2.750) -- (6.250,2.750);
  \draw[line width = 0.2, color = black, densely dotted] (6.250,3.000) -- (6.500,3.000) -- (6.500,3.250) -- (6.250,3.250);
  \draw[line width = 0.2, color = black, densely dotted] (6.250,3.500) -- (6.500,3.500) -- (6.500,3.750) -- (6.250,3.750);
  \draw[line width = 0.2, color = black, densely dotted] (6.250,4.000) -- (6.500,4.000) -- (6.500,4.250) -- (6.250,4.250);
  \draw[line width = 0.2, color = black, densely dotted] (6.250,4.500) -- (6.500,4.500) -- (6.500,4.750) -- (6.250,4.750);
  \draw[line width = 0.2, color = black, densely dotted] (6.250,5.000) -- (6.500,5.000) -- (6.500,5.250) -- (6.250,5.250);
  \draw[line width = 0.2, color = black, densely dotted] (6.250,5.500) -- (6.500,5.500) -- (6.500,5.750) -- (6.250,5.750);

  \draw[line width = 0.01cm, color = black] (0.750,0.750) -- (0.750,0.500) -- (1.000,0.500)-- (1.000,0.750) -- (1.250,0.750) -- (1.250,0.500) -- (1.500,0.500)-- (1.500,0.750) -- (1.750,0.750) -- (1.750,0.500) -- (2.000,0.500)-- (2.000,0.750) -- (2.250,0.750) -- (2.250,0.500) -- (2.500,0.500)-- (2.500,0.750) -- (2.750,0.750) -- (2.750,0.500) -- (3.000,0.500)-- (3.000,0.750) -- (3.250,0.750) -- (3.250,0.500) -- (3.500,0.500)-- (3.500,0.750) -- (3.750,0.750) -- (3.750,0.500) -- (4.000,0.500)-- (4.000,0.750) -- (4.250,0.750) -- (4.250,0.500) -- (4.500,0.500)-- (4.500,0.750) -- (4.750,0.750) -- (4.750,0.500) -- (5.000,0.500)-- (5.000,0.750) -- (5.250,0.750) -- (5.250,0.500) -- (5.500,0.500)-- (5.500,0.750) -- (5.750,0.750) -- (5.750,0.500) -- (6.000,0.500) -- (6.000,0.750);

  \draw[line width = 0.01cm, color = black] (0.750,6.000) -- (0.750,6.250) -- (1.000,6.250)-- (1.000,6.000) -- (1.250,6.000) -- (1.250,6.250) -- (1.500,6.250)-- (1.500,6.000) -- (1.750,6.000) -- (1.750,6.250) -- (2.000,6.250)-- (2.000,6.000) -- (2.250,6.000) -- (2.250,6.250) -- (2.500,6.250)-- (2.500,6.000) -- (2.750,6.000) -- (2.750,6.250) -- (3.000,6.250)-- (3.000,6.000) -- (3.250,6.000) -- (3.250,6.250) -- (3.500,6.250)-- (3.500,6.000) -- (3.750,6.000) -- (3.750,6.250) -- (4.000,6.250)-- (4.000,6.000) -- (4.250,6.000) -- (4.250,6.250) -- (4.500,6.250)-- (4.500,6.000) -- (4.750,6.000) -- (4.750,6.250) -- (5.000,6.250)-- (5.000,6.000) -- (5.250,6.000) -- (5.250,6.250) -- (5.500,6.250)-- (5.500,6.000) -- (5.750,6.000) -- (5.750,6.250) -- (6.000,6.250) -- (6.000,6.000);

  \draw[line width = 0.01cm, color = black] (0.750,0.750) -- (0.500,0.750) -- (0.500,1.000)-- (0.750,1.000) -- (0.750,1.250) -- (0.500,1.250) -- (0.500,1.500)-- (0.750,1.500) -- (0.750,1.750) -- (0.500,1.750) -- (0.500,2.000)-- (0.750,2.000) -- (0.750,2.250) -- (0.500,2.250) -- (0.500,2.500)-- (0.750,2.500) -- (0.750,2.750) -- (0.500,2.750) -- (0.500,3.000)-- (0.750,3.000) -- (0.750,3.250) -- (0.500,3.250) -- (0.500,3.500)-- (0.750,3.500) -- (0.750,3.750) -- (0.500,3.750) -- (0.500,4.000)-- (0.750,4.000) -- (0.750,4.250) -- (0.500,4.250) -- (0.500,4.500)-- (0.750,4.500) -- (0.750,4.750) -- (0.500,4.750) -- (0.500,5.000)-- (0.750,5.000) -- (0.750,5.250) -- (0.500,5.250) -- (0.500,5.500)-- (0.750,5.500) -- (0.750,5.750) -- (0.500,5.750) -- (0.500,6.000) -- (0.750,6.000);

  \draw[line width = 0.01cm, color = black] (6.000,0.750) -- (6.250,0.750) -- (6.250,1.000)-- (6.000,1.000) -- (6.000,1.250) -- (6.250,1.250) -- (6.250,1.500)-- (6.000,1.500) -- (6.000,1.750) -- (6.250,1.750) -- (6.250,2.000)-- (6.000,2.000) -- (6.000,2.250) -- (6.250,2.250) -- (6.250,2.500)-- (6.000,2.500) -- (6.000,2.750) -- (6.250,2.750) -- (6.250,3.000)-- (6.000,3.000) -- (6.000,3.250) -- (6.250,3.250) -- (6.250,3.500)-- (6.000,3.500) -- (6.000,3.750) -- (6.250,3.750) -- (6.250,4.000)-- (6.000,4.000) -- (6.000,4.250) -- (6.250,4.250) -- (6.250,4.500)-- (6.000,4.500) -- (6.000,4.750) -- (6.250,4.750) -- (6.250,5.000)-- (6.000,5.000) -- (6.000,5.250) -- (6.250,5.250) -- (6.250,5.500)-- (6.000,5.500) -- (6.000,5.750) -- (6.250,5.750) -- (6.250,6.000) -- (6.000,6.000);

\end{tikzpicture}

%% file: fig-cube-1.tex
\begin{tikzpicture}

  \path[use as bounding box] (-0.250,-0.250) rectangle (7.000,7.000);

  \draw[line width = 0.01cm, color = black] (0.750,0.750) -- (0.750,0.500) -- (1.000,0.500)-- (1.000,0.250) -- (1.250,0.250) -- (1.250,0.500) -- (1.500,0.500)-- (1.500,0.750) -- (1.750,0.750) -- (1.750,0.500) -- (2.000,0.500)-- (2.000,0.750) -- (2.250,0.750) -- (2.250,0.500) -- (2.500,0.500)-- (2.500,0.250) -- (2.750,0.250) -- (2.750,0.500) -- (3.000,0.500)-- (3.000,0.750) -- (3.250,0.750) -- (3.250,0.500) -- (3.500,0.500)-- (3.500,0.250) -- (3.750,0.250) -- (3.750,0.500) -- (4.000,0.500)-- (4.000,0.750) -- (4.250,0.750) -- (4.250,0.500) -- (4.500,0.500)-- (4.500,0.250) -- (4.750,0.250) -- (4.750,0.500) -- (5.000,0.500)-- (5.000,0.250) -- (5.250,0.250) -- (5.250,0.500) -- (5.500,0.500)-- (5.500,0.250) -- (5.750,0.250) -- (5.750,0.500) -- (6.000,0.500) -- (6.000,0.750);
  \draw[line width = 0.2, color = black, densely dotted] (4.750,0.250) -- (4.750,0.000) -- (5.000,0.000) -- (5.000,0.250);
  \draw[line width = 0.2, color = black, densely dotted] (5.250,0.250) -- (5.250,0.000) -- (5.500,0.000) -- (5.500,0.250);
  \draw[line width = 0.01cm, color = black] (0.750,6.000) -- (0.750,6.250) -- (1.000,6.250)-- (1.000,6.000) -- (1.250,6.000) -- (1.250,6.250) -- (1.500,6.250)-- (1.500,6.000) -- (1.750,6.000) -- (1.750,6.250) -- (2.000,6.250)-- (2.000,6.000) -- (2.250,6.000) -- (2.250,6.250) -- (2.500,6.250)-- (2.500,6.500) -- (2.750,6.500) -- (2.750,6.250) -- (3.000,6.250)-- (3.000,6.000) -- (3.250,6.000) -- (3.250,6.250) -- (3.500,6.250)-- (3.500,6.500) -- (3.750,6.500) -- (3.750,6.250) -- (4.000,6.250)-- (4.000,6.000) -- (4.250,6.000) -- (4.250,6.250) -- (4.500,6.250)-- (4.500,6.000) -- (4.750,6.000) -- (4.750,6.250) -- (5.000,6.250)-- (5.000,6.500) -- (5.250,6.500) -- (5.250,6.250) -- (5.500,6.250)-- (5.500,6.500) -- (5.750,6.500) -- (5.750,6.250) -- (6.000,6.250) -- (6.000,6.000);
  \draw[line width = 0.2, color = black, densely dotted] (5.250,6.500) -- (5.250,6.750) -- (5.500,6.750) -- (5.500,6.500);
  \draw[line width = 0.01cm, color = black] (0.750,0.750) -- (0.500,0.750) -- (0.500,1.000)-- (0.250,1.000) -- (0.250,1.250) -- (0.500,1.250) -- (0.500,1.500)-- (0.750,1.500) -- (0.750,1.750) -- (0.500,1.750) -- (0.500,2.000)-- (0.750,2.000) -- (0.750,2.250) -- (0.500,2.250) -- (0.500,2.500)-- (0.250,2.500) -- (0.250,2.750) -- (0.500,2.750) -- (0.500,3.000)-- (0.750,3.000) -- (0.750,3.250) -- (0.500,3.250) -- (0.500,3.500)-- (0.250,3.500) -- (0.250,3.750) -- (0.500,3.750) -- (0.500,4.000)-- (0.750,4.000) -- (0.750,4.250) -- (0.500,4.250) -- (0.500,4.500)-- (0.250,4.500) -- (0.250,4.750) -- (0.500,4.750) -- (0.500,5.000)-- (0.250,5.000) -- (0.250,5.250) -- (0.500,5.250) -- (0.500,5.500)-- (0.750,5.500) -- (0.750,5.750) -- (0.500,5.750) -- (0.500,6.000) -- (0.750,6.000);
  \draw[line width = 0.2, color = black, densely dotted] (0.250,4.750) -- (0.000,4.750) -- (0.000,5.000) -- (0.250,5.000);
  \draw[line width = 0.01cm, color = black] (6.000,0.750) -- (6.250,0.750) -- (6.250,1.000)-- (6.500,1.000) -- (6.500,1.250) -- (6.250,1.250) -- (6.250,1.500)-- (6.500,1.500) -- (6.500,1.750) -- (6.250,1.750) -- (6.250,2.000)-- (6.500,2.000) -- (6.500,2.250) -- (6.250,2.250) -- (6.250,2.500)-- (6.000,2.500) -- (6.000,2.750) -- (6.250,2.750) -- (6.250,3.000)-- (6.500,3.000) -- (6.500,3.250) -- (6.250,3.250) -- (6.250,3.500)-- (6.000,3.500) -- (6.000,3.750) -- (6.250,3.750) -- (6.250,4.000)-- (6.500,4.000) -- (6.500,4.250) -- (6.250,4.250) -- (6.250,4.500)-- (6.500,4.500) -- (6.500,4.750) -- (6.250,4.750) -- (6.250,5.000)-- (6.000,5.000) -- (6.000,5.250) -- (6.250,5.250) -- (6.250,5.500)-- (6.000,5.500) -- (6.000,5.750) -- (6.250,5.750) -- (6.250,6.000) -- (6.000,6.000);
  \draw[line width = 0.2, color = black, densely dotted] (6.500,1.250) -- (6.750,1.250) -- (6.750,1.500) -- (6.500,1.500);
  \draw[line width = 0.2, color = black, densely dotted] (6.500,1.750) -- (6.750,1.750) -- (6.750,2.000) -- (6.500,2.000);
  \draw[line width = 0.2, color = black, densely dotted] (6.500,4.250) -- (6.750,4.250) -- (6.750,4.500) -- (6.500,4.500);

\end{tikzpicture}